\documentclass[12pt,reqno]{amsart}
\usepackage{amssymb, amsfonts, amsbsy, latexsym, epsfig, color}
\usepackage{amsmath} 
\usepackage{amsthm}
\usepackage{amsmath,amssymb}
\usepackage{url}

\newcommand{\vertiii}[1]{{\left\vert\kern-0.25ex\left\vert\kern-0.25ex\left\vert #1
    \right\vert\kern-0.25ex\right\vert\kern-0.25ex\right\vert}}
\newtheorem{Thm}{Theorem}[section]

\newtheorem{cor}[Thm]{Corollary}
\newtheorem{lemma}[Thm]{Lemma}

\newtheorem{proposition}[Thm]{Proposition}
\newtheorem{definition}[Thm]{Definition}

\newtheorem{theorem}[Thm]{Theorem}

\newcommand\cE{{\mathcal{E}}}

\newcommand\cH{{\mathcal{H}}}

\newcommand{\bitem}{\begin{itemize}}
\newcommand{\eitem}{\end{itemize}}
\newcommand{\benum}{\begin{enumerate}}
\newcommand{\eenum}{\end{enumerate}}
\newcommand{\beq}{\begin{equation}}
\newcommand{\eeq}{\end{equation}}
\newcommand{\ip}[2]{\langle#1,#2\rangle}

  \newcommand{\R}{\mathbb{R}}
 \newcommand{\N}{\mathbb{N}}
 
 \newcommand{\Z}{\mathbb{Z}}

\DeclareMathOperator*{\supp}{supp}

\def\RR{\mathbb{R}}

\def\cH{{\mathcal{H}}}

\def\cH{\mathcal{H}}

\newcommand{\WQ}[1]{#1}

\begin{document}
\title{Dualizable Shearlet Frames and Sparse Approximation}

\author[G. Kutyniok]{Gitta Kutyniok}
\address{Department of Mathematics, Technische Universit\"at Berlin, 10623 Berlin, Germany}
\email{kutyniok@math.tu-berlin.de}

\author[W.-Q Lim]{Wang-Q Lim}
\address{Department of Mathematics, Technische Universit\"at Berlin, 10623 Berlin, Germany}
\email{lim@math.tu-berlin.de}

\thanks{G.K. acknowledges support by the Einstein Foundation Berlin, by the Einstein Center for Mathematics Berlin
(ECMath), by Deutsche Forschungsgemeinschaft (DFG) Grant KU 1446/14, by the DFG Collaborative Research Center TRR 109
``Discretization in Geometry and Dynamics'', and by the DFG Research Center {\sc Matheon} ``Mathematics for key technologies''
in Berlin. Parts of the research for this paper was performed while the first author was visiting the Department of
Mathematics at the ETH Z\"urich. G.K. thanks this department for its hospitality and support during this visit.
W.L. would like to thank the DFG Collaborative Research Center TRR 109 ``Discretization in Geometry and Dynamics''
and the DFG Research Center {\sc Matheon} ``Mathematics for key technologies'' in Berlin for its support.}

\begin{abstract}
Shearlet systems have been introduced as directional representation systems, which provide optimally sparse
approximations of a certain model class of functions governed by anisotropic features while allowing faithful
numerical realizations by a unified treatment of the continuum and digital realm. They are redundant systems,
and their frame properties have been extensively studied. In contrast to certain band-limited shearlets,
compactly supported shearlets provide high spatial localization, but do not constitute Parseval frames. Thus
reconstruction of a signal from shearlet coefficients requires knowledge of a dual frame. However, no closed
and easily computable form of any dual frame is known.

In this paper, we introduce the class of dualizable shearlet systems, which consist of compactly supported
elements and can be proven to form frames for $L^2(\R^2)$. For each such dualizable shearlet system, we then provide an explicit
construction of an associated dual frame, which can be stated in closed form and efficiently computed. We also
show that dualizable shearlet frames still provide optimally sparse approximations of anisotropic features.
\end{abstract}


\maketitle

\section{Introduction}\label{sec:intro}

During the last years, methodologies utilizing sparse approximations have had a tremendous impact on
data science. This is foremost due to the method of compressed sensing (see \cite{CRT06,Don06c} or
\cite{DDEK12}), which played a major role in the initiation of today's paradigm that any type of data
admits a sparse representation within a suitably chosen orthonormal basis, or, more generally, a
frame \cite{Chr03}. In fact, frames -- redundant, yet stable systems -- are typically preferable due to the
added flexibility the redundancy provides. However, although a frame might provide even
optimally sparse approximations within a model situation, in the end, one still needs to reconstruct
the data from the respective frame coefficients. For an orthonormal basis, this can be achieved by
the classical decomposition formula. In the situation of a frame though, a so-called dual frame is
required.

In this paper, we will consider this problem in the situation of imaging sciences. Since it is
typically assumed that images are governed by edge-like structures, a common model situation are
cartoon-like functions, which are -- coarsely speaking -- compactly supported piecewise smooth functions.
Shearlet systems \cite{KL12}, which might be among the most widely used directional representation
systems today, have been shown to deliver optimally sparse approximations of this class. However,
their compactly supported version, though superior due to high spatial localization, forms a (non-tight) frame
for $L^2(\R^2)$; but the construction of a dual having a closed and easily computable form is an
open problem.

\subsection{Data Processing by Frames}

Frames have a long history in providing decompositions and expansions for data processing, and the
reader might consult \cite{CK12} for applications in audio processing or communication theory.
A {\em frame} for a Hilbert space $\cH$ is a sequence $(\varphi_i)_{i \in I}$ satisfying $A \|f\|^2
\le \sum_{i \in I} |\langle f , \varphi_i \rangle |^2 \le B \|f\|^2$ for all $f \in \cH$ with
$0 < A \le B < \infty$. If  the frame bounds $A$ and $B$ can chosen to be equal, it is typically
called {\em tight frame}; in case of $A=B=1$, a {\em Parseval frame}.

Analysis of an element $f \in \cH$ by a frame $(\varphi_i)_{i \in I}$ is typically achieved by application
of the {\em analysis operator} $T$ given by
\[
T : \cH \to \ell^2(I), \quad f \mapsto (\langle f,\varphi_i \rangle)_{i \in I}.
\]
Reconstruction of $f$ from the sequence of {\em frame coefficients} $(\langle f,\varphi_i \rangle)_{i \in I}$
is possible by utilizing the adjoint operator $T^*$, since it can be shown that
\beq \label{eq:reconstructionformula}
f = \sum_{i \in I} \ip{f}{\varphi_i} (T^*T)^{-1}\varphi_i \quad \mbox{for all } f \in \cH.
\eeq
Unless $(\varphi_i)_{i \in I}$ forms a tight frame -- in which case $T^*T$ is a multiple of the identity --,
we face the difficulty to have to invert the operator $T^*T$ in order to compute the {\em canonical dual frame}
$((T^*T)^{-1}\varphi_i)_{i \in I}$.

In fact, certainly, the canonical dual frame is not the only choice for deriving a reconstruction
formula such as \eqref{eq:reconstructionformula}. In general, one calls $(\tilde{\varphi}_i)_{i \in I}$
an associated {\em dual frame}, if the following is true:
\beq \label{eq:expansion_true}
f = \sum_{i \in I} \ip{f}{\varphi_i} \tilde{\varphi}_i \quad \mbox{for all } f \in \cH.
\eeq

\subsection{Sparse Approximation using Frames}

One key feature of frames, which is in particular beneficial for deriving sparse approximations, is their
redundancy. Sparse approximation by a frame $(\varphi_i)_{i \in I}$ can be regarded from two sides: On the
one hand, we might consider expansions in terms of the frame such as
\beq \label{eq:expansion}
f = \sum_{i \in I} c_i \varphi_i,
\eeq
expecting the existence of some coefficient sequence $(c_i)_{i \in I}$, which is sparse in the sense of,
for instance, $\|(c_i)_{i \in I}\|_{\ell^1(I)} < \infty$ or at least $\|(c_i)_{i \in I}\|_{\ell^p(I)} < \infty$
for some $p < 2$.

This is however not the approach normally taken in data science, in particular, related to compressed sensing.
Instead we expect that the sequence of frame coefficients $(\langle f,\varphi_i \rangle)_{i \in I}$ is sparse.
In \cite{NDEG12}, this situation is termed {\em co-sparsity}, and in fact sparsity within a frame is typically
exploited in this way. For instance, reconstruction from highly undersampled data is then achieved by placing
the $\ell_1$-norm on such coefficient sequences and mimimizing over all $f \in \cH$.

Thus, instead of expansions of the form \eqref{eq:expansion}, we consider \eqref{eq:expansion_true} in the sense
of a reconstruction procedure. This certainly requires having access to some dual frame associated with
$(\varphi_i)_{i \in I}$. One can circumvent this problem by using iterative methods such as conjugate
gradients whose efficiency depends heavily on the ratio of the frame bounds. But such methods deliver
only approximate solutions and are rather slow.

\subsection{Imaging Science and Anisotropic Features}

Images play a key part in data science as a significant percentage of data today are in fact images. Following
the program discussed before, it is illusory to assume that reasonable results can be derived for the whole
Hilbert space $L^2(\R^2)$. Hence we restrict to an appropriate subset, which models features images are
assumed to be governed by.

As such a class typically so-called {\em cartoon-like functions} introduced in \cite{Don01} are taken,
which are basically compactly supported functions that are $C^2$ apart from a closed $C^2$ discontinuity
curve with bounded curvature (Definition \ref{defi:cartoon}). The intuition is that edge-like structures
are typically prominent in images and, in addition, the neurons in the visual cortex of humans also react very
strongly to those features. It should be emphasized that certainly such structures appear in
other situations as well such as in solutions of transport dominated equations \cite{DHSW12,DKLSW15}.

Donoho then proved in \cite{Don01} that the $L^2$-error of best $N$-term approximation $f_N$ of such a
cartoon-like function $f$ by any frame for $L^2(\R^2)$ behaves as
\[
\|f-f_N\|_2 \gtrsim N^{-1}  \qquad \text{as } N \rightarrow \infty.
\]
This results provides a notion of optimality, and frames satisfying this approximation rate up to a
$\log$-factor are customarily referred to as systems delivering {\em optimal sparse approximations}
within the class of cartoon-like functions.

\subsection{Shearlet Systems} \label{subsec:shearlets}

Shearlet systems were originally introduced in \cite{GKL06} as a directional representation system which
meets this benchmark result, but which -- in contrast to the previously advocated system of curvelets
\cite{CD04} -- fit into the framework of affine systems and also allow a faithful implementation by a unified
treatment of the continuum and digital realm.

Shearlet systems are based on three operations: {\em parabolic scaling} $A_j$, $j \in \Z$ to provide different
resolution levels, {\em shearing} $S_k$, $k \in \Z$ to provide different orientations, both given by
\[
A_j = \left(
  \begin{array}{cc}
    2^j & 0 \\
    0 & \WQ{2^{\lfloor\frac{j}{2}\rfloor}} \\
  \end{array}
\right)
\quad \text{and} \quad
S_k = \left(
  \begin{array}{cc}
    1 & k \\
    0 & 1 \\
  \end{array}
\right),
\]
as well as {\em translation} to provide different positions. The definition of a (cone-adapted) shearlet system is
then as follows. We wish to mention that the term ``cone-adapted'' is due to the fact that the different systems
$\Psi(\psi; c)$ and $\tilde{\Psi}(\tilde{\psi}; c)$ are responsible for the horizontal and vertical cone in Fourier domain, respectively;
thereby, together with $\Phi(\phi; c_1)$ achieving a complete system with a finite set of shears
for each sale $j$.

\begin{definition}
For $\phi, \psi, \tilde{\psi} \in L^2(\R^2)$ and $c=(c_1,c_2) \in (\R_+)^2$, the \emph{(cone-adapted) shearlet system}
$\mathcal{SH}(\phi,\psi,\tilde{\psi}; c)$ is defined by
\[
\mathcal{SH}(\phi,\psi,\tilde{\psi}; c) = \Phi(\phi; c_1) \cup \Psi(\psi; c) \cup \tilde{\Psi}(\tilde{\psi}; c),
\]
where
\begin{align*}
  \Phi(\phi; c_1) &= \{\phi_m := \phi(\cdot- c_1 m) : m \in \Z^2\},\\
  \Psi(\psi; c) &= \{\psi_{j,k,m} := \WQ{|\mathrm{det}(A_j)|^{1/2}} \psi(S_k A_{j}\cdot  - \mathrm{diag}(c_1,c_2) m) : j \ge 0,  |k| \leq \WQ{2^{\lceil j/2 \rceil}}, m \in \Z^2\},\\
  \tilde{\Psi}(\tilde{\psi}; c) &= \{\tilde{\psi}_{j,k,m} :=  \WQ{|\mathrm{det}(\tilde{A}_j)|^{1/2}} \tilde{\psi}(S_k^T\tilde{A}_{j}\cdot  -\mathrm{diag}(c_2,c_1) m): j \ge 0,
 |k| \leq \WQ{2^{\lceil j/2 \rceil}}, m \in \Z^2 \},
\end{align*}
with $\tilde{A}_{j} = \mathrm{diag}(2^{\WQ{\lfloor j/2 \rfloor}},2^j)$.
\end{definition}

\subsection{Problems with Shearlet Frames}

For high spatial localization, compactly supported shearlet systems are considered, which are also implemented
in ShearLab (see \url{www.ShearLab.org}) \cite{KLR14}. As shown in \cite{KKL12}, compactly supported generators
$\phi, \psi, \tilde{\psi} \in L^2(\R^2)$ can be constructed such that the associated shearlet system
$\mathcal{SH}(\phi,\psi,\tilde{\psi}; c)$ constitutes a frame -- not a tight frame -- for $L^2(\R^2)$ with
controllable frame bounds dependent on $c$. Under slightly stronger conditions, it was proven in \cite{KL11} that
such systems also deliver optimally sparse approximations of cartoon-like functions.

In the situation of bandlimited shearlet frames (i.e., the Fourier transform is compactly supported), Grohs derived
results on the existence of ``nice'' duals \cite{Gro13}. However, for compactly supported shearlet frames
no closed, easily computable form of an associated dual frame is known, even when allowing small modifications of
the shearlet system.

\subsection{Our Contribution}

In this paper, we present a solution to this problem. We construct a shearlet system which can be regarded as
being of the form $\mathcal{SH}(\phi,\psi,\tilde{\psi}; c)$ and satisfies the following properties:
\bitem
\item The system is compactly supported and forms a frame for $L^2(\R^2)$.
\item The system delivers optimal sparse approximations of cartoon-like functions.
\item An associated dual frame can be stated in closed form and efficiently computed.
\item It is composed of orthonormal bases, which provides it with a distinct, accessible structure.
\eitem

In addition, the novel proof technique which we use for proving the approximation properties of dualizable shearlet
frames along the way allow an improvement of previous approximation results from \cite{KL11} for the class of compactly
supported shearlet frames introduced in \cite{KKL12} with respect to the exponent of the additional $\log$-term (see
Corollary \ref{cor:approx}). It should be mentioned that with this result, this exponent in the $\log$-term is the
smallest known for any directional representation system, in particular, curvelets \cite{CD04}.

\subsection{Outline}

The paper is organized as follows. The construction of dualizable shearlet systems is presented and discussed in
Section \ref{sec:construction}; the definition is stated in Definition \ref{def:dualizableshearlets}. Section
\ref{sec:dual} is devoted to the analysis of frame properties of dualizable shearlet systems, namely showing (in
Theorem \ref{theo:dualframe}) that these systems do form frames for $L^2(\R^2)$ and that an associated dual frame
can be explicitly given in closed form. The statement that dualizable shearlet systems do provide optimally sparse
approximations of cartoon-like functions is presented in Section \ref{sec:sparse} as Theorem \ref{thm:sparsity}.
Since the proof is very involved, the key steps and the core part are presented in Section \ref{sec:proofMainTheorem}
with the proofs of several preliminary lemmata being outsourced to Section \ref{sec:proofs}.

\section{Construction of Dualizable Shearlet Frames}\label{sec:construction}

This section is devoted to the construction of dualizable shearlet frames. One key ingredient is a family
of orthonormal bases for each shearing direction, which will be discussed in Subsection \ref{subsec:basic}.
Since those elements lack directionality in the sense of wedgelike shape elements, they are subsequently
filtered, yielding the desired dualizable shearlet systems (see Subsection \ref{subsec:definition}). We
emphasize that we will only present the construction for the horizontal Fourier cone in detail -- compare
the cone-based definition of shearlet systems from Subsection \ref{subsec:shearlets} --; the vertically
aligned system will be derived by switching the variables, or rotation by 90 degrees denoted by $R_{\frac{\pi}{2}}$.

Since the construction is rather technical in nature, it is not initially clear that the term ``shearlets'' is
justified; and we argue in Subsection \ref{subsec:comparison} in favor of this expression by comparing the novel
systems to customarily defined shearlets (cf. Subsection \ref{subsec:shearlets}).

\subsection{Basic Ingredients}\label{subsec:basic}

To construct a family of orthonormal bases for $L^2(\R^2)$ for each shearing direction, we first let
$\varphi_1, \psi_1 \in L^2(\R)$ be compactly supported functions, which satisfy the support condition
\beq \label{eq:support_varphi_1}
\delta_{\varphi_1} = \inf_{\xi \in [-\frac12,\frac12]} |\hat\varphi_1(\xi)| > 0,
\eeq
as well as, for some $\rho \in (0,\frac{2}{13})$, $\alpha \ge \frac{6}{\rho}+1$, and $\beta > \alpha + 1$, the decay conditions
\beq \label{eq:decay_varphi_1}
\Bigl|\Bigl(\frac{d}{d\xi}\Bigr)^{\ell}\hat\psi_1(\xi)\Bigr| \lesssim \frac{\min\{1,|\xi|^{\alpha}\}}{(1+|\xi|)^{\beta}}
\quad \mbox{and} \quad
\Bigl|\Bigl(\frac{d}{d\xi}\Bigr)^{\ell}\hat\varphi_1(\xi)\Bigr| \lesssim \frac{1}{(1+|\xi|)^{\beta}}
\quad \mbox{for } \ell=0,1.
\eeq
In addition, we require the system
\[
\{\varphi_1(\cdot-m) : m \in \Z\} \cup \{2^{j/2}\psi_1(2^j\cdot-m) : j \ge 0, m \in \Z\}
\]
to form an orthonormal basis for $L^2(\R)$. For the existence of such functions, we refer to \cite{Dau92}.

Second, we utilize this univariate system to construct the desired family of orthonormal bases. We now follow
the following strategy: We lift the system to $L^2(\R^2)$ in such a way that we achieve a tiling of Fourier
domain according to Figure \ref{fig:ortho_tile}(a), and then apply shearing operators.

To generate the anticipated tiling, for $x = (x_1,x_2) \in \R^2$, we set
\[
\psi^0(x) := \psi_1(x_1)\varphi_1(x_2) \quad \mbox{and} \quad
\psi^{p}(x) := 2^{(p-1)/2}\psi_1(x_1)\psi_1(2^{p-1}x_2) \; \mbox{for } p > 0
\]
as well as
\[
\varphi^0(x) := \varphi_1(x_1)\varphi_1(x_2) \quad \mbox{and} \quad
\varphi^{p}(x) := 2^{(p-1)/2}\varphi_1(x_1)\psi_1(2^{p-1}x_2) \; \mbox{for } p > 0.
\]
Notice that the parameter $p$ will be utilized to derive the dyadic substructure in vertical direction.

For a fixed integer $j_0 \ge 0$, we next consider the system given by
\[
\{|\mathrm{det}(A_{j_0})|^{1/2}\varphi^p(A_{j_0} \cdot-D_pm), \:
|\mathrm{det}(A_j)|^{1/2}\psi^p(A_j \cdot-D_pm) : j \ge j_0, p \ge 0\},
\]
where $D_p = \mathrm{diag}(1,2^{-\max\{p-1,0\}})$, which achieves the tiling of Fourier domain as depicted in Figure \ref{fig:ortho_tile}(a). It will be
shown in Lemma \ref{lemm:ONB} that this system forms an orthonormal basis for $L^2(\R^2)$.

We now carefully insert the shearing operator. Drawing from the definition of ``normal'' shearlets,
the shear parameter should equal $\WQ{\frac{k}{2^{\lceil j/2 \rceil}}}$ with \WQ{$|k| \leq 2^{\lceil j/2\rceil }$}. Since we later need to
parameterize by those quotients, we require a unique representation without ambiguity. For this,
we define the map
\[
s : \{(j,q) : j = 0, q = 0\} \cup \{j : j \ge 0\} \times \{q : \WQ{|q| \leq 2^{j}}, \, q \in 2 \Z +1\} \to \WQ{[-1,1]}, \,\, s(j,q) := \WQ{\frac{q}{2^{j}}},
\]
which is obviously injective. Thus from now on, we consider the set of shear parameters given by
\[
\mathbb{S} = \{s(j,q) = 0 : j = 0, q = 0\} \cup \{s(j,q) : j \ge 0,\, \WQ{|q| \leq 2^{j}}, \, q \in 2 \Z +1\}.
\]

Armed with definition, we can now define what we call shearlet-type wavelet systems.

\begin{definition} \label{def:shearlettypewavelets}
Let $\varphi_1, \psi_1 \in L^2(\R)$, and let $\varphi^p, \psi^p \in L^2(\R^2)$, $p \ge 0$ be defined as before.
Further, set $D_p := \mathrm{diag}(1,2^{-d_p})$ with $d_p := \max\{p-1,0\}$. Then, for each shear parameter \WQ{$s := s(\lceil j_0/2\rceil,q_0) \in \mathbb{S}$, where $j_0$ is the smallest nonnegative integer such that $s = \frac{q_0}{2^{\lceil j_0/2\rceil}}$}, we define
the {\em shearlet-type wavelet system} $\Psi_s(\varphi_1, \psi_1)$ by
\[
\Psi_s(\varphi_1, \psi_1) := \{\varphi_{j_0,s,m,p}, \psi_{j,s,m,p} : j \ge j_0, m \in \Z^2, p \ge 0\},
\]
where
\[
\varphi_{j,s,m,p} := |\mathrm{det}(A_j)|^{1/2}\varphi^p(A_jS_s\cdot-D_p m) \;\: \mbox{ and } \;\: \psi_{j,s,m,p} := |\mathrm{det}(A_j)|^{1/2}\psi^p(A_jS_s\cdot-D_p m).
\]
\end{definition}

The tiling of Fourier domain by shearlet-type wavelet systems is depicted in Figure \ref{fig:ortho_tile}.
\begin{figure}[htb]
\begin{center}
\hspace*{\fill}
\includegraphics[width=0.4\textwidth]{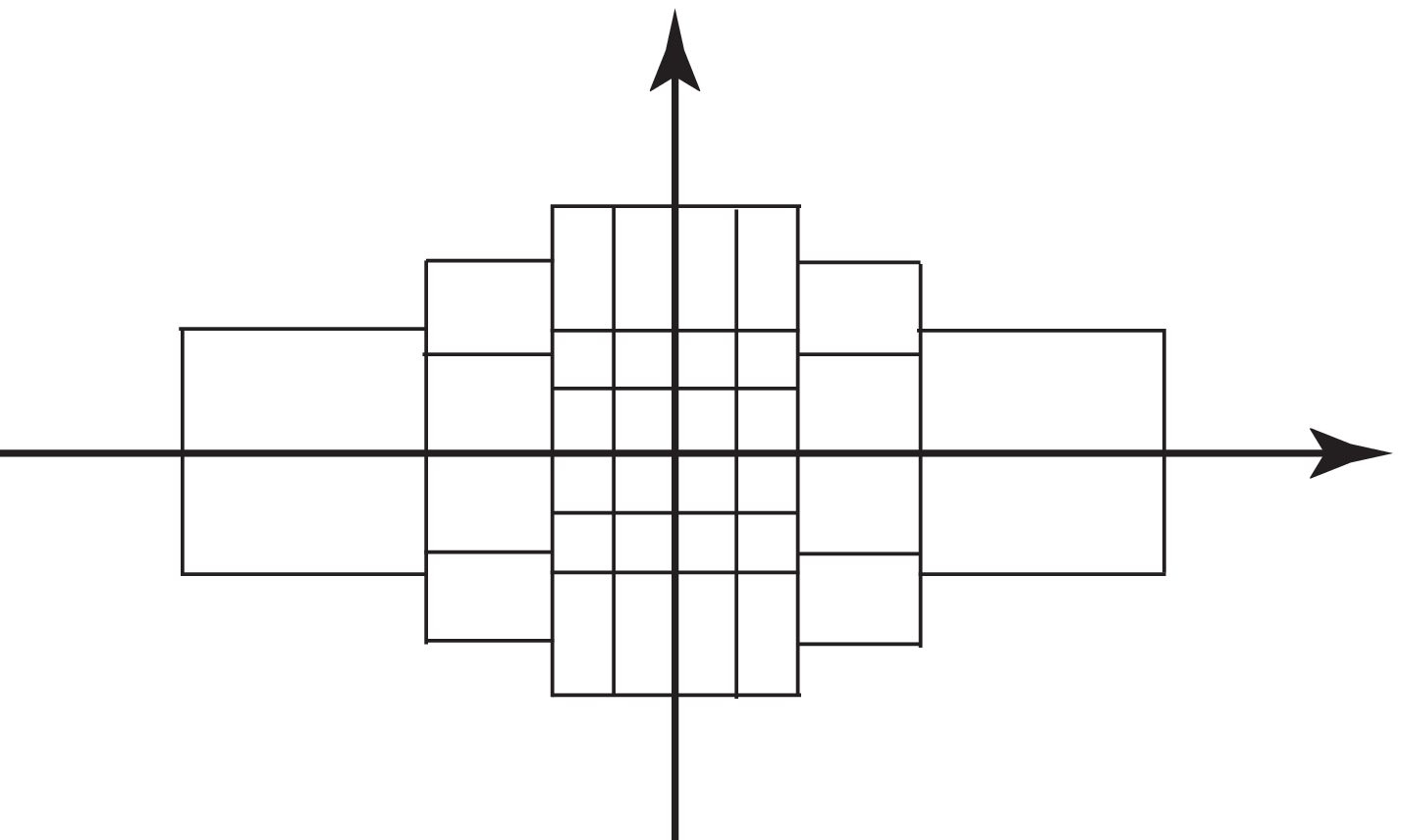}
\hfill
\includegraphics[width=0.4\textwidth]{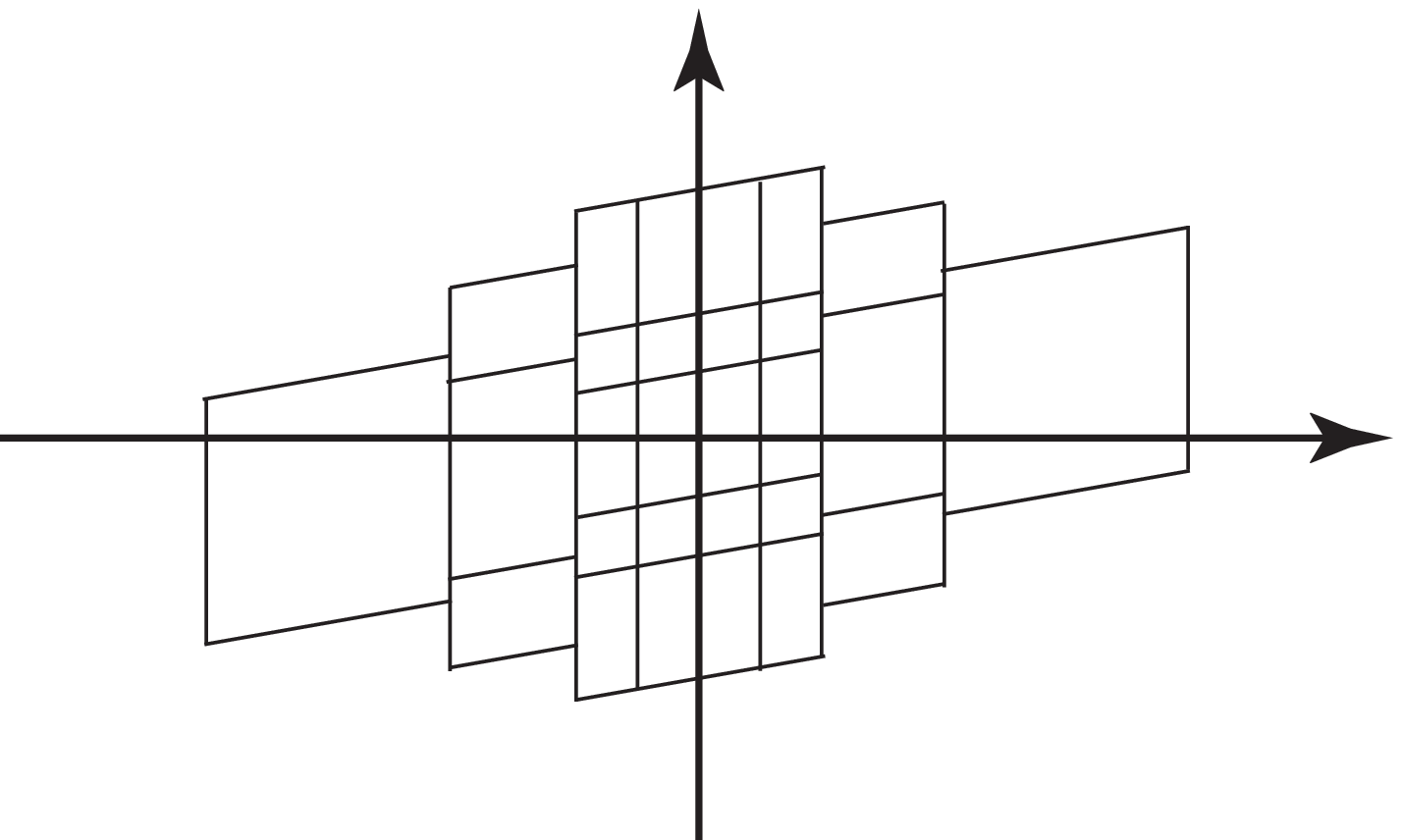}
\hspace*{\fill}
\put(-354,-15){{\scriptsize{(a)}}}
\put(-130,-15){{\scriptsize{(b)}}}
\end{center}
\caption{Tiling of Fourier domain by $\Psi_s$ with (a) $s = 0$ and (b) $s = \frac14$.}
\label{fig:ortho_tile}
\end{figure}

The next result shows that, for each shear parameter, the associated shearlet-type wavelet system indeed
constitutes an orthonormal basis.

\begin{lemma}\label{lemm:ONB}
For each $s \in \mathbb{S}$, the shearlet-type wavelet system $\Psi_s(\varphi_1, \psi_1)$ is
an orthonormal basis for $L^2(\R^2)$.
\end{lemma}

\begin{proof}
Without loss of generality, we consider $\Psi_{0}(\varphi_1,\psi_1)$ where $s = 0$ (and hence $j_0 = 0$) in Definition \ref{def:shearlettypewavelets}.
Then, for $x = (x_1,x_2) \in \R^2$, by definition,
\[
\varphi_{0,0,m,0}(x) = \varphi_1(x_1-m_1)\varphi_1(x_2-m_2) \; \mbox{ and } \; \varphi_{0,0,m,p}(x) = \varphi_1(x_1-m_1)2^{\frac{d_p}{2}}\psi_1(2^{d_p}x_2-m_2)
\]
as well as, in addition for $j \ge 0$ and $p > 0$,
\[
\psi_{j,0,m,0}(x) = 2^{\frac{j}{2}}\psi_1(2^jx_1-m_1)\WQ{2^{\frac12\lfloor \frac{j}{2} \rfloor}\varphi_1(2^{\lfloor \frac{j}{2} \rfloor}x_2-m_2)}
\]
and
\[
\psi_{j,0,m,p}(x) = \WQ{2^{\frac{j}{2}}\psi_1(2^jx_1-m_1)2^{\frac{d_p}{2}+\frac{1}{2}\lfloor \frac{j}{2}\rfloor}\psi_1(2^{d_p}(2^{\lfloor\frac{j}{2}\rfloor}x_2)-m_2)}.
\]
Next, for each $j \ge 0$, let $V_j$ and $W_j$ be the subspaces of $L^2(\R^2)$ defined by
\[
V_j = \overline{\mbox{span}\{2^{\frac{j}{2}}\varphi_1(2^j\cdot-m) : m \in \Z\}}
\quad \mbox{and} \quad
W_j = \overline{\mbox{span}\{2^{\frac{j}{2}}\psi_1(2^j\cdot-m) : m \in \Z\}}.
\]
By construction, for each $p > 0$, the systems $\{\varphi_{0,0,m,0} : m \in \Z^2\}$ and
$\{\varphi_{0,0,m,p} : m \in \Z^2\}$ form orthonormal bases for $V_0 \otimes V_0$ and $V_0 \otimes W_{d_p}$, respectively.
Similarly, for $j \ge 0$ and $p > 0$, $\{\psi_{j,0,m,0} : m \in \Z^2\}$ and $\{\psi_{j,0,m,p} : m \in \Z^2\}$ form
orthonormal bases for $W_j \otimes V_0$ and $W_j \otimes W_{d_p}$, respectively.

Since $V_0 \bot W_j$ and $W_j \bot W_{j'}$ for $j, j' \ge 0$, $j \neq j' $, those subspaces are mutually orthogonal and,
for each $j \ge 0$,
\[
\Bigl( V_0 \otimes V_0 \Bigr) \oplus \Bigl( \bigoplus_{p = 1}^{\infty} V_0 \otimes W_{d_p} \Bigr) \oplus \Bigl( W_j \otimes V_0 \Bigr)
\oplus \Bigl( \bigoplus_{p = 1}^{\infty} W_j \otimes W_{d_p}\Bigr) = (V_0 \otimes L^2(\R)) \oplus (W_j \otimes L^2(\R)).
\]
From this, we finally obtain
\[
\Bigl(V_0 \otimes L^2(\R)\Bigr) \oplus \Bigl(\bigoplus_{j = 0}^{\infty} W_j \otimes L^2(\R)\Bigr) = L^2(\R^2).
\]
This proves our claim.
\end{proof}

We wish to mention that the definition of a dualizable shearlet system in Definition \ref{def:dualizableshearlets}
will also require the systems derived by switching the variable in $\Psi_s(\varphi_1, \psi_1)$, i.e., by rotation by $R_{\frac{\pi}{2}}$.

\subsection{Definition}\label{subsec:definition}

The next step consists of a filtering procedure. To define the filters, let $g \in L^2(\R^2)$ be a
compactly supported function satisfying the conic support condition
\beq \label{eq:support_g}
\delta_g = \inf_{\xi \in \Omega_{g}} |\hat g(\xi)| > 0, \quad \mbox{where } \Omega_{g} = \{\xi \in \R^2 : |\tfrac{\xi_2}{\xi_1}| < 1, \tfrac12 < |\xi_1| < 1\Big\},
\eeq
as well as the decay condition
\beq \label{eq:decay_g}
\Bigl|\Bigl(\frac{\partial}{\partial \xi_2}\Bigr)^{\ell}\hat g(\xi)\Bigr| \lesssim \frac{\min\{1,|\xi_1|^{\alpha}\}}{(1+|\xi_1|)^{\beta}(1+|\xi_2|)^{\beta}}
\quad \mbox{for } \ell=0,1,
\eeq
with $\alpha$ and $\beta$ chosen as before (i.e., $\rho \in (0,\frac{2}{13})$, $\alpha \ge \frac{6}{\rho}+1$, and $\beta > \alpha + 1$).

At this point, we pause in the description of the construction, and first observe the following frame-type
equation, which follows from our choices. Notice that this result already combines systems for the horizontal
and vertical cones. For the proof, we refer to \cite{KKL12}.

\begin{lemma}[\cite{KKL12}]
\label{lemm:framePhiPsi}
Letting $\varphi_1, \psi_1 \in L^2(\R)$ and $g \in L^2(\R^2)$ be defined as before, we have
\[
0 < A \leq |\hat \varphi^0(\xi)|^2 + \sum_{j \ge 0}\sum_{\WQ{|k| \leq 2^{\lceil j/2 \rceil}}} |\hat g(S^{-T}_kA^{-1}_j \xi)|^2
+ |\hat{g}( S^{-T}_k {A}^{-1}_j R_{\frac{\pi}{2}}\xi)|^2 \leq  B < \infty \;\: \mbox{for a.e. } \xi \in \R^2,
\]
where $\varphi^0(x) = \varphi_1(x_1)\varphi_1(x_2)$.
\end{lemma}

The filters $G_s$, $s = \WQ{s(\lceil j_0 /2\rceil,q_0)} \in \mathbb{S}$   are then defined by
\beq \label{eq:defiGs}
\hat G_{0}(\xi) = |\hat \varphi^0(\xi)|^2+\sum_{j = 0}^{\infty} |\hat g(A^{-1}_jS^{-T}_s \xi)|^2
\quad \mbox{and} \quad
\hat G_{s}(\xi) = \sum_{j = j_0}^{\infty} |\hat g(A^{-1}_jS^{-T}_s \xi)|^2 \quad \mbox{for } s \neq 0.
\eeq
Figure \ref{fig:shear_tree} illustrates the frequency tiling by the essential supports of $\hat G_s$, showing the
wedgelike shape geometry.
\begin{figure}[htb]
\begin{center}
\includegraphics[width=0.3\textwidth]{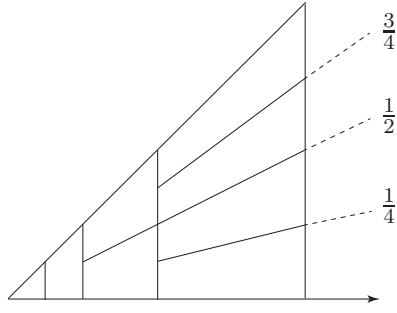}
\put(0,34){{\small{$\frac14$}}}
\put(0,68){{\small{$\frac12$}}}
\put(0,100){{\small{$\frac34$}}}
\end{center}
\caption{Frequency tiling with $\hat G_s$ for $s = 0,\frac14,\frac12,\frac34$ (for $\xi_1 > 0$).}
\label{fig:shear_tree}
\end{figure}

The following result provides an identity, which will be a key ingredient to prove the frame property of
the dualizable shearlet systems in Theorem \ref{theo:dualframe}.

\begin{lemma} \label{lemm:filterequation}
Let $G_s$, $s \in \mathbb{S}$ be defined as in \eqref{eq:defiGs}. Then
\[
\sum_{s \in \mathbb{S}} \hat G_s(\xi) = |\hat \varphi^0(\xi)|^2 + \sum_{j \ge 0}\sum_{|k| \WQ{\leq 2^{\lceil j/2 \rceil} }} |\hat g(S_k^{-T}A^{-1}_j\xi)|^2
\quad \mbox{for a.e. } \xi \in \R^2.
\]
\end{lemma}

\begin{proof}
This follows directly from the definition of the filters and the set $\mathbb{S}$.
\end{proof}

Finally, after this preparation, we can now formally define dualizable shearlet systems by filtering
the shearlet-type wavelet systems defined in Definition \ref{def:shearlettypewavelets} using the
filters $G_s$, $s \in \mathbb{S}$.

\begin{definition} \label{def:dualizableshearlets}
For any $s \in \mathbb{S}$, let $\Psi_s(\varphi_1, \psi_1)$ be the shearlet-type wavelet system, and let $G_s$ be the
filter generated by $g \in L^2(\R^2)$ as defined in \eqref{eq:defiGs}. Then
the {\em dualizable shearlet system} $\mathcal{SH}(\varphi_1, \psi_1; g)$ is defined by
\[
\mathcal{SH}(\varphi_1, \psi_1; g) = \{\psi^{\ell}_{\lambda} : \lambda \in \Lambda_s, s \in \mathbb{S}, \ell = 0,1\},
\]
with index set
\[
\Lambda_{s} = \{(j,s,m,p) : j \in \{-1\} \cup \{j_0, j_0+1, \ldots\}, m \in \Z^2, p \ge 0\} \quad \mbox{for } s =  \WQ{s(\lceil j_0 /2\rceil,q_0)},
\]
where
\[
\psi^{0}_{\lambda} =
\left\{
\begin{array}{rcl}
G_s * \varphi_{j_0,s,m,p} & : & \lambda = (-1,s,m,p) \in \Lambda_s,\\
G_s * \psi_{j,s,m,p} & : &  \lambda = (j,s,m,p) \in \Lambda_s,
\end{array}
\right.
\]
and $\psi^{1}_{\lambda} = \psi^{0}_{\lambda} \circ R_{\frac{\pi}{2}}$.
\end{definition}

We immediately observe that the constructed system is compactly supported.

\begin{lemma}\label{lem:compact}
Each dualizable shearlet system is compactly supported.
\end{lemma}

\begin{proof}
Let $\mathcal{SH}(\varphi_1, \psi_1; g)$ be a dualizable shearlet system. Then by construction, there exist
$C_1, C_2 > 0$ such that, for all $s =  \WQ{s(\lceil j_0 /2\rceil,q_0)} \in \mathbb{S}$ and $j \ge j_0, m \in \Z^2, p \ge 0$,
the filters $G_s$ are compactly supported with
\[
\mathrm{supp}(G_s) \subset S_s^{-1}A_{j_0}^{-1}[-C_1,C_1]^2
\]
and the elements of the shearlet-type wavelet systems satisfy
\[
\mathrm{supp}(\varphi_{j_0,s,m,p}) \subset S_s^{-1}A_{j_0}^{-1}[-C_2,C_2]^2,\,\, \mathrm{supp}(\psi_{j,s,m,p}) \subset S_s^{-1}A_{j}^{-1}[-C_2,C_2]^2.
\]
Hence, there exists some $C > 0$ such that
\[
\mathrm{supp}(G_s*\varphi_{j_0,s,m,p}),\: \mathrm{supp}(G_s*\psi_{j,s,m,p}) \subset S_s^{-1}A_{j_0}^{-1}[-C,C]^2,
\]
which proves the claim.
\end{proof}

\subsection{Comparison with Customarily Defined Shearlet Systems}\label{subsec:comparison}

We now aim to justify the term ``shearlets'' by rewriting the elements of $\mathcal{SH}(\varphi_1, \psi_1; g)$
such that the resemblance with cone-adapted shearlet systems (cf. Subsection \ref{subsec:shearlets}) is revealed.
We will observe that the dualizable shearlet system consists of functions of the form contained in the original shearlet system
except for the oversampling matrix $D_p$ for $p \ge 0$. As already mentioned before, this ingredient ensures that
a dualizable shearlet system is composed of subsystems, which are filtered versions of orthonormal bases. This
structure will be key to have a closed form for an associated dual frame (see Theorem \ref{theo:dualframe}).

\begin{proposition}\label{prop:comparison}
Let $\mathcal{SH}(\varphi_1, \psi_1; g)$ be a dualizable shearlet system. Define
\[
\hat \Theta(\xi) = |\hat{\varphi^0}(\xi)|^2+\sum_{j = 0}^{\infty} |\hat g(A^{-1}_j \xi)|^2
\quad \mbox{and} \quad
\hat \Theta_{\ell}(\xi) = \sum_{j = -\ell}^{\infty} |\hat g(A^{-1}_j \xi)|^2, \quad \ell \ge 0.
\]
Then, for the elements of $\mathcal{SH}(\varphi_1, \psi_1; g)$, the following hold.
\begin{enumerate}
\item[(i)] For all $m \in \Z^2$ and $p \ge 0$,
\[
\psi^0_{\lambda} = (\Theta*\varphi^p)(\cdot - D_p m) \quad \mbox{with}\,\, \lambda = (-1,0,m,p) \in \Lambda_0
\]
and for all $j \ge 0$, $m \in \Z^2$, and $p \ge 0$,
\[
\psi^0_{\lambda} = |\det(A_j)|^{1/2}(\Theta*\psi^p)(A_j\cdot - D_p m) \quad \mbox{with}\,\, \lambda = (j,0,m,p) \in \Lambda_0
\]
\item[(ii)] Letting $s =  \WQ{s(\lceil j_0 /2\rceil,q_0)} \in \mathbb{S} \setminus \{0\}$, for all $k \in \Z$ with $\WQ{s = \tfrac{k}{2^{\lceil j_0/2 \rceil}}}$ and for all $m \in \Z^2$, $p \ge 0$,
\[
\psi^0_{\lambda} = |\det(A_{j_0})|^{1/2}(\Theta_{0}*\varphi^p)(S_kA_{j_0}\cdot - D_p m) \quad \mbox{with}\,\, \lambda = (-1,s,m,p) \in \Lambda_s
\]
and, for all $j \ge j_0$ and $k \in \Z$ with $s = \WQ{\tfrac{k}{2^{\lceil j/2 \rceil}}}$, and for all $m \in \Z^2$, $p \ge 0$,
\[
\psi^0_{\lambda} = |\det(A_j)|^{1/2}(\Theta_{j-j_0}*\psi^p)(S_kA_j\cdot - D_p m) \quad \mbox{with}\,\, \lambda = (j,s,m,p) \in \Lambda_s
\]
\end{enumerate}
\end{proposition}

\begin{proof}
We will only consider the last equation in (ii) for $\psi^0_{\lambda}$. The other cases can be derived similarly with minor modifications for notation.
First note that for each $(j,k) \in \N_0 \times \WQ{\{-2^{\lceil j/2 \rceil},\dots, 2^{\lceil j/2 \rceil}\} \backslash \{0\}}$, there exists a
unique shear parameter $\WQ{s(\lceil j_0 /2\rceil,q_0)} \in \mathbb{S}$ with $j \ge j_0$ and $k = \WQ{(2^{\lceil \frac{j}{2} \rceil -\lceil \frac{j_0}{2} \rceil})q_0}$. This ensures
\[
A_j^{-1}S^{-T}_{s} = S^{-T}_{k}A^{-1}_j.
\]
Using this relation, we obtain
\begin{eqnarray*}
\widehat{G_s * \psi_{j,s,m,p}}(\xi) &=& \sum_{j^{'} = j_0}^{\infty}|\hat g(A^{-1}_{j^{'}}S^{-T}_{s}\xi)|^2|\mathrm{det}(A_j)|^{-1/2}\hat\psi^p(A^{-1}_{j}S^{-T}_{s}\xi)
e^{-2\pi i \ip{m}{D_p A^{-1}_{j}S^{-T}_{s}\xi}} \\
&=& |\mathrm{det}(A_j)|^{-1/2}\hat\Theta_{j-j_0}(S^{-T}_kA^{-1}_{j}\xi)\hat\psi^p(S^{-T}_kA^{-1}_{j}\xi)e^{-2\pi i \ip{m}{D_p S^{-T}_kA^{-1}_{j}\xi}}.
\end{eqnarray*}
Application of the inverse Fourier transform and careful consideration of the different cases yield the claim.
\end{proof}

\section{The Dual of Dualizable Shearlet Frames}\label{sec:dual}

Dualizable shearlet systems are foremost designed to provide a closed, easily computable form of a dual
frame while still delivering optimally sparse approximation of cartoon-like functions. The first item
will now be formally stated and proved.

\begin{theorem}\label{theo:dualframe}
Let $\mathcal{SH}(\varphi_1, \psi_1; g) = \{\psi^{\ell}_{\lambda} : \lambda \in \Lambda_s, s \in \mathbb{S}, \ell = 0,1\}$ be
a dualizable shearlet system, which constitutes a frame for $L^2(\R^2)$. Then
\[
\widetilde{\mathcal{SH}}(\varphi_1, \psi_1; g) = \{\tilde{\psi}^{\ell}_{\lambda} : \lambda \in \Lambda_s, s \in \mathbb{S}, \ell = 0,1\}
\]
is a dual frame for $\mathcal{SH}(\varphi_1, \psi_1; g)$, where, for $\lambda \in \Lambda_s$,
\[
\hat{\tilde{\psi}}^0_{\lambda} = \frac{\hat{G}_s\cdot \hat{\psi}^0_{\lambda}}{\sum_{s' \in \mathbb{S}}|\hat{G}_{s'}|^2+|\hat{G}_{s'} \circ R_{\frac{\pi}{2}}|^2}
\quad \mbox{and} \quad
\tilde{\psi}^1_{\lambda} = \tilde{\psi}^0_{\lambda} \circ R_{\frac{\pi}{2}}.
\]
\end{theorem}

\begin{proof}
For the proof, we use the convention that $\overline{g}(\cdot) = g(-\cdot)$. We first observe that the structure of
a dualizable shearlet system allows a decomposition as
\begin{eqnarray*}
\sum_{s \in \mathbb{S}}\sum_{\lambda \in \Lambda_s}|\langle f,\psi^0_{\lambda}\rangle|^2 = & = & \sum_{s=\WQ{s(\lceil j_0 /2\rceil,q_0)} \in \mathbb{S}}\Bigl( \sum_{m \in \Z^2}\sum_{p \in \N_0} |\langle \overline{G}_s*f,\varphi_{j_0,\WQ{s(\lceil j_0 /2\rceil,q_0)},m,p}\rangle|^2 \\
&& + \sum_{j = j_0}^{\infty} \sum_{m \in \Z^2}\sum_{p \in \N_0} |\langle \overline{G}_s*f,\psi_{j,\WQ{s(\lceil j_0 /2\rceil,q_0)},m,p}\rangle|^2 \Bigr).
\end{eqnarray*}
Using the orthonormal basis property proven in Lemma \ref{lemm:ONB}, we can conclude that
\beq \label{eq:proofdual_1}
\sum_{s \in \mathbb{S}}\sum_{\lambda \in \Lambda_s}|\langle f,\psi^0_{\lambda}\rangle|^2
= \int_{\R^2} \Bigl( \sum_{s \in \mathbb{S}}|\hat{G}_s(\xi)|^2\Bigr)|\hat{f}(\xi)|^2 d\xi.
\eeq
Similarly, we can show that
\beq \label{eq:proofdual_2}
\sum_{s \in \mathbb{S}}\sum_{\lambda \in \Lambda_s}|\langle f,\psi^1_{\lambda}\rangle|^2
= \int_{\R^2} \Bigl( \sum_{s \in \mathbb{S}}|\hat{G}_s(R_{\frac{\pi}{2}} \xi)|^2\Bigr)|\hat{f}(\xi)|^2 d\xi.
\eeq
By Lemmata \ref{lemm:framePhiPsi} and \ref{lemm:filterequation}, it follows that
\[
\sum_{s \in \mathbb{S}}|\hat{G}_s(\xi)|^2+|\hat{G}_s(R_{\frac{\pi}{2}} \xi)|^2 \leq 2B^2 < \infty  \quad \mbox{for a.e. } \xi.
\]
Combining this inequality with \eqref{eq:proofdual_1} and \eqref{eq:proofdual_2} implies the existence of an upper frame
bound for $\mathcal{SH}(\varphi_1, \psi_1; g)$.

To derive a lower frame bound, we use the support conditions on $\varphi_1$ and $g$, namely \eqref{eq:support_varphi_1}
and \eqref{eq:support_g}, which imply
\begin{eqnarray*}
\lefteqn{\sum_{s \in \mathbb{S}}|\hat{G}_s(\xi)|^2+|\hat{G}_s(R_{\frac{\pi}{2}} \xi)|^2}\\
& \ge & \Bigl(\chi_{[-1/2,1/2]^2}+\sum_{j \ge 0}\sum_{|k| \WQ{\leq 2^{\lceil j/2 \rceil}}} \chi_{\WQ{S^{T}_{k/2^{j/2}}A_{j}}\Omega_{g}}
+ \chi_{\WQ{S^{T}_{k/2^{j/2}}A_{j}}\Omega_{g}} \circ R_{\frac{\pi}{2}}\Bigr)\cdot (\min\{\delta^2_{\varphi_1},\delta_{g}\})^2 \\
&\ge& (\min\{\delta^2_{\varphi_1},\delta_{g}\})^2  > 0 \quad \mbox{for a.e. } \xi.
\end{eqnarray*}

The frame property of $\widetilde{\mathcal{SH}}(\varphi_1, \psi_1; g)$ can be shown by similar arguments.

It remains to prove that $\widetilde{\mathcal{SH}}(\varphi_1, \psi_1; g)$ forms indeed a dual frame of $\mathcal{SH}(\varphi_1, \psi_1; g)$.
For this, we use the structure of the system $\widetilde{\mathcal{SH}}(\varphi_1, \psi_1; g)$ to obtain
\begin{eqnarray*}
\sum_{s \in \mathbb{S}}\sum_{\lambda \in \Lambda_s}\langle \hat f, \hat{\psi}^{0}_{\lambda}\rangle\hat{\tilde{\psi}}^{0}_{\lambda}
& = & \hspace*{-0.5cm} \sum_{s=\WQ{s(\lceil j_0 /2\rceil,q_0)} \in \mathbb{S}} \frac{\hat{G}_{s}}{\sum_{s^{'}\in \mathbb{S}} |\hat{G}_{s^{'}}|^2+|\hat{G}_{s^{'}} \circ R_{\frac{\pi}{2}}|^2}
\Bigl(\sum_{m \in \Z^2}\sum_{p \in \N_0}\langle \overline{\hat{G}}_s\hat{f},\hat{\varphi}_{j_0,s,m,p}\rangle \hat{\varphi}_{j_0,s,m,p} \\
&& + \sum_{j = j_0}^{\infty} \sum_{m \in \Z^2}\sum_{p \in \N_0} \langle \overline{\hat{G}}_s\hat{f},\hat{\psi}_{j,s,m,p}\rangle \hat{\psi}_{j,s,m,p}\Bigr).
\end{eqnarray*}
Lemma \ref{lemm:ONB} again implies
\beq \label{eq:proofdual_3}
\sum_{s \in \mathbb{S}}\sum_{\lambda \in \Lambda_s}\langle \hat f, \hat{\psi}^{0}_{\lambda}\rangle\hat{\tilde{\psi}}^{0}_{\lambda}
= \sum_{s \in \mathbb{S}} \frac{|\hat{G}_s|^2 \cdot \hat{f}}{\sum_{s^{'}\in \mathbb{S}} |\hat{G}_{s^{'}}|^2 + |\hat{G}_{s^{'}}\circ R_{\frac{\pi}{2}}|^2}.
\eeq
Similarly,
\beq \label{eq:proofdual_4}
\sum_{s \in \mathbb{S}}\sum_{\lambda \in \Lambda_s}\langle \hat f, \hat{\psi}^{1}_{\lambda}\rangle\hat{\tilde{\psi}}^{1}_{\lambda}
= \sum_{s \in \mathbb{S}} \frac{|\hat{G}_s\circ R_{\frac{\pi}{2}}|^2 \cdot \hat{f}}{\sum_{s^{'}\in \mathbb{S}} |\hat{G}_{s^{'}}|^2 + |\hat{G}_{s^{'}}\circ R_{\frac{\pi}{2}}|^2}.
\eeq
Using the filter properties as well as \eqref{eq:proofdual_3} and \eqref{eq:proofdual_4} finally yields
\[
\sum_{\ell=0}^{1}\sum_{s\in \mathbb{S}}\sum_{\lambda\in \Lambda_s}\langle \hat f, \hat{\psi}^{\ell}_{\lambda}\rangle\hat{\tilde{\psi}}^{\ell}_{\lambda} = \hat{f}.
\]
The theorem is proved.
\end{proof}

We remark that the dual frame does not form a (dualizable) shearlet system. However, this was also not to be
expected, since already for wavelet frames, only very few dual frames do have the form of a wavelet system \cite{BL07}.

\section{Sparse Approximation Properties}\label{sec:sparse}

We now turn to analyzing sparse approximation properties of dualizable shearlet systems with respect to anisotropic
features which are modeled by the class of cartoon-like functions. We start by formally introducing this
class, which was first defined in \cite{Don01}. We remark that the superscript $2$ in the notion $\cE^2(\RR^2)$
is due to the fact that the discontinuity curve is assumed to be $C^2$. Generalizations of cartoon-like functions
with different types of regularity can, for instance, be found in \cite{KLL12}.

\begin{definition} \label{defi:cartoon}
The set of {\it cartoon-like functions} $\cE^2(\RR^2)$ is defined by
\[
\cE^2(\RR^2) = \{f \in L^2(\RR^2) : f = f_0 + f_1 \cdot \chi_{B}\},
\]
where $B \subset [0,1]^2$ is a nonempty, simply connected set with $C^2$-boundary, $\partial B$ has bounded curvature,
and $f_i \in C^2(\RR^2)$ satisfies $\mathrm{supp}  f_i \subseteq [0,1]^2$ and $\|f_i\|_{C^2} \le 1$ for $i=0, 1$.
\end{definition}
We now let
\[
\Lambda := \{0,1\} \times \bigcup_{s \in \mathbb{S}} \Lambda_s
\]
which is the index set for $\mathcal{SH}(\varphi_1,\psi_1;g)$.
Given a dualizable shearlet system
\[
\mathcal{SH}(\varphi_1, \psi_1; g) = \{\psi^{\ell}_{\lambda} : \lambda \in \Lambda_s, s \in \mathbb{S}, \ell = 0,1\}
\]
with associated dual frame
\[
\widetilde{\mathcal{SH}}(\varphi_1, \psi_1; g) = \{\tilde{\psi}^{\ell}_{\lambda} : \lambda \in \Lambda_s, s \in \mathbb{S}, \ell = 0,1\}
\]
as defined in Theorem \ref{theo:dualframe}, we are interested in $N$-term approximations of $f \in \cE^2(\RR^2)$ of the form
\[
f_N = \sum_{(\ell,\lambda) \in \Lambda_N} \langle f,\psi^{\ell}_{\lambda}\rangle \tilde{\psi}^{\ell}_{\lambda},
\]
where $\Lambda_N \subseteq \Lambda$, $\# \Lambda_N = N$. Let us remind the reader
that we choose expansions in terms of the dual frame, since applications usually require reconstruction from the frame
coefficients $(\langle f,\psi^{\ell}_{\lambda}\rangle)_{\ell, \lambda}$ (cf. Section \ref{sec:intro}).

Without loss of generality, we will only consider shearlet elements $\psi^{0}_{\lambda} \in \mathcal{SH}(\varphi_1, \psi_1; g)$
associated with one frequency cone. Since the elements $\psi^{1}_{\lambda}$ just arise as rotation by 90 degrees, they can be dealt
with similarly. Hence, for the sake of brevity, we from now on omit the superscript ``0'', i.e., we write
\[
\psi_{\lambda} := \psi^{0}_{\lambda} \quad \mbox{and} \quad \tilde{\psi}_{\lambda} := \tilde{\psi}^{0}_{\lambda}.
\]

Next we recall that the optimally achievable approximation rate, i.e., a benchmark for any frame for $L^2(\R^2)$, was also derived in \cite{Don01}.

\begin{theorem}[\cite{Don01}]
Let $(h_i)_{i \in I} \subseteq L^2(\R^2)$ be a frame for $L^2(\R^2)$. Then, for any $f \in \cE^2(\RR^2)$, the $L^2$-error of
best $N$-term approximation by $f_N$ with respect to $(h_i)_{i \in I}$ satisfies
\[
\|f-f_N\|_2 \gtrsim N^{-1}  \qquad \text{as } N \rightarrow \infty.
\]
\end{theorem}

The following result shows that the approximation rate of dualizable shearlets for cartoon-like functions can be
arbitrarily close to the optimal rate as the smoothness of the generators is increased, i.e., as $\rho \to 0$.

\begin{theorem}\label{thm:sparsity}
Let $\mathcal{SH}(\varphi_1, \psi_1; g) = \{\psi^{\ell}_{\lambda} : \lambda \in \Lambda_s, s \in \mathbb{S}, \ell = 0,1\}$ be
a dualizable shearlet system, and let $\rho$ be the smoothness parameter defined in \eqref{eq:decay_varphi_1} and \eqref{eq:decay_g}
associated with $\varphi_1, \psi_1, g$. Further, let $f \in \mathcal{E}^2(\R^2)$. Then
\[
\|f - f_N\|_2 \lesssim N^{-1+\frac{15}{2}\rho} \cdot \log(N) \qquad \text{as } N \rightarrow \infty,
\]
where $f_N = \sum_{(\ell,\lambda) \in \Lambda_N} \langle f,\psi^{\ell}_{\lambda}\rangle \tilde{\psi}^{\ell}_{\lambda}$ with
$\Lambda_N \subseteq \Lambda$, $\# \Lambda_N = N$ is the $N$-term approximation
using the $N$ largest coefficients $(\langle f,\psi^{\ell}_{\lambda}\rangle)_{\ell,\lambda}$.
\end{theorem}

Before we discuss the overall structure and details of the proof, we would like to highlight that using this
new proof technique, even previous results can be improved. In fact, we can lower the exponent of the $\log$-factor
in the decay rate of the compactly supported shearlet system defined in \cite{KL11} from $\log(N)^{3/2}$ to
$\log(N)$.

\begin{cor} \label{cor:approx}
For each $f \in \mathcal{E}^2(\R^2)$, the compactly supported shearlet system defined in \cite{KL11} provides an approximation rate of
\[
\|f - f_N\|_2 \lesssim N^{-1} \cdot \log(N) \quad \text{as } N \rightarrow \infty
\]
with $f_N$ being the $N$-term approximation consisting of the $N$ largest shearlet coefficients.
\end{cor}

The proof of this corollary follows the proof of Theorem \ref{thm:sparsity} quite closely except for slight
modifications which we describe in Subsection \ref{subsec:technical_0}.

\section{Proof of Theorem \ref{thm:sparsity}}\label{sec:proofMainTheorem}

Since the proof is rather technical and complex, we start by discussing its overall architecture. We recall from
Proposition \ref{prop:comparison} that for all $\lambda = (j,s,m,p) \in \Lambda_s$, $s = \WQ{s(\lceil j_0 /2\rceil,q_0)} \in \mathbb{S}$,
$j \ge j_0$ and $k \in \Z$ with \WQ{$s = \tfrac{k}{2^{\lceil j/2 \rceil}}$},
\beq \label{eq:formclassical}
\psi_{\lambda} = |\det(A_j)|^{1/2}(\Theta_{j-j_0}*\psi^p)(S_kA_j\cdot - D_p m).
\eeq
We mention that without loss of generality, we only need to consider shearlet elements of this form. Nearly
identical arguments can be applied for the elements $\psi_{\lambda}$ with $\lambda = (-1,s,m,p) \in \Lambda_s$
with minor modifications for notation.

One might think that due to the fact that dualizable shearlets have this strong structural similarity with ``normal''
shearlets, the steps of the proof of the (optimal sparse) approximation result from \cite{KL11} could be directly
applied. This is however not the case. Although, in the end, we will be able to utilize some of those steps,
careful preparation for this is required. Moreover, it will turn out that we will eventually even improve
the approach from \cite{KL11} in the sense of Corollary \ref{cor:approx}, i.e., by reducing the number of
$\log$-factors.

In a first step, we prove two basic estimates for the shearlet coefficients, namely for an $L^\infty$ and for an
$L^2$ function. This is made precise in the following lemma, whose technically natured proof can be found in
Subsection \ref{subsec:technical_2}.

\begin{lemma} \label{lemm:technical_2}
\begin{enumerate}
\item[(i)] For $f \in L^{\infty}(\R^2)$, we have
\[
|\langle f,\psi_{\lambda}\rangle| \lesssim 2^{-\frac{3}{4}j} \cdot \|f\|_{\infty} \quad \mbox{for all } \lambda = (j,s,m,p) \in \Lambda.
\]
\item[(ii)]  For $f \in L^2(\R^2)$, we have
\[
|\langle f,\psi_{\lambda}\rangle| \lesssim 2^{-\frac{\alpha}{2} p} \cdot \|f\|_2 \quad \text{for all } \lambda = (j,s,p,m) \in \Lambda.
\]
\end{enumerate}
\end{lemma}

One main difficulty in proving this result is the analysis of the function $\Theta_{j-j_0}*\psi^p$, which is now
the generator of the shearlet element in \eqref{eq:formclassical}. In fact, we require a universal upper bound
for this function, which is given by the following result. For its proof, we refer to Subsection \ref{subsec:technical_1}.

\begin{lemma} \label{lemm:technical_1}
There exists a universal constant $C$ such that
\[
\|\Theta_{j-j_0}*\psi^p\|_1 \le C \quad \mbox{for all } j \ge j_0, p \ge 0.
\]
\end{lemma}

Aiming to drive an at least similar overall strategy as in the proof of \cite{KL11}, let us recall the hypotheses
from this paper. One key condition is that customarily defined shearlets generated by $A_j$ and $S_k$ are supported
in $S^{-1}_sA^{-1}_j[-C,C]^2$ for some $C > 0$ with $\WQ{s = \frac{k}{2^{\lceil j/2 \rceil}}}$. In this area, the $C^2$ curvilinear
singularity of a cartoon-like function is well approximated by its tangent. However, Lemma \ref{lem:compact} shows
that for dualizable shearlets, we can only conclude that $\mbox{supp}(\psi_{\lambda}) \subset S^{-1}_sA^{-1}_{j_0}[-C,C]^2$.
It can be computed that $\mbox{supp}(\psi_{\lambda})$ is in fact essentially the same as the support of customarily
defined shearlets for scale $j = j_0$. However, $\mbox{supp}(\psi_{\lambda})$ is much larger when $j \gg j_0$.
To resolve this issue for our newly defined dualizable shearlets $\psi_{\lambda}$, we will approximate $\psi_{\lambda}$
by more suitable functions of smaller supports comparable to the size of the supports of ``normal'' shearlets with
controllable error bound. This is the essence of the following result, whose proof is outsourced to Subsection
\ref{subsec:technical_3}. For the following lemmata, the parameter $\rho$ defined in \eqref{eq:decay_varphi_1} will be used.

\begin{lemma}\label{lemm:technical_3}
Let $\lambda = (j,s,m,p) \in \Lambda_s$ with $s = \WQ{s(\lceil j_0 /2\rceil,q_0)} \in \mathbb{S}$, $j \ge j_0$ and set
\[
\hat{\psi}^{\sharp}_{\lambda} := \sum_{j^{'} = \max\{\lfloor j(1-\rho) \rfloor,j_0\}}^{\infty}|\hat g (A^{-1}_{j^{'}}S^{-T}_s \xi)|^2 \hat\psi_{j,s,m,p}.
\]
Then the following hold.
\begin{itemize}
\item[(i)] There exists some $C > 0$ such that
\[
\supp(\psi^{\sharp}_{\lambda}) \subset S^{-1}_s A^{-1}_j \Bigl([-2^{j\rho}C,2^{j\rho}C] \times [-2^{j\rho/2}C, 2^{j\rho/2}C]+m\Bigr).
\]
\item[(ii)] We have
\[
|\langle f,(\psi_{\lambda}-\psi^{\sharp}_{\lambda})\rangle| \lesssim 2^{-\frac{j}{2}\rho \alpha} \|f\|_2 \quad \text{for all } f \in L^2(\R^2).
\]
\end{itemize}
\end{lemma}

The second key condition is a directional vanishing moment condition, which can be shown to be fulfilled by the generators
$\Theta_{j-\max\{\lfloor j(1-\rho) \rfloor,j_0\}}*\psi^p$ for $\psi^{\sharp}_{\lambda}$. In fact, following the same argument in the proof
of Proposition \ref{prop:comparison}(ii), we see that
\[
\psi^{\sharp}_{\lambda} = |\mbox{det}(A_j)|^{1/2}(\Theta_{j-\max\{\lfloor j(1-\rho) \rfloor,j_0\}}*\psi^p)(S_kA_j \cdot - D_p m)
\]
for $\lambda = (j,s,m,p) \in \Lambda_s$, $s = \WQ{s(\lceil j_0 /2\rceil,q_0)} \in \mathbb{S}$, $j \ge j_0$ and $k \in \Z$ with $\WQ{s = \tfrac{k}{2^{\lceil j/2 \rceil}}}$.
The proof of the following result is provided in Subsection \ref{subsec:technical_4}.

\begin{lemma}\label{lemm:technical_4}
For all $j \ge j_0$, $p \ge 0$, $\ell = 0,1$ and \WQ{$\gamma \in [1,\alpha)$},
\[
\Bigl|\Bigl(\frac{\partial}{\partial \xi_2}\Bigr)^{\ell}\widehat{(\Theta_{j-\max\{\lfloor j(1-\rho) \rfloor,j_0\}}*\psi^p)}(\xi)\Bigr|
\lesssim \frac{\min\{1,|\xi_1|\}^{\alpha-1}}{(1+|\xi_1|)^{\beta-\gamma}(1+|\xi_2|)^{\gamma}}.
\]
\end{lemma}

One main last ingredient, which we state as a lemma before providing the complete proof of Theorem \ref{thm:sparsity}, are
decay rate of the shearlet coefficients $\langle f,\psi_{\lambda}\rangle$ for cartoon-like functions, where we now
carefully insert conditions related to the functions $\psi^{\sharp}_{\lambda}$. Again, the proof can be found in
Subsection \ref{subsec:technical_5}.

\begin{lemma}\label{lemm:technical_5}
Assume that $f \in \mathcal{E}^2(\R^2)$ with $C^2$ discontinuity curve given by $x_1 = E(x_2)$. For $\psi_{\lambda} \in \mathcal{SH}(\varphi_1,\psi_1;g)$
with $\lambda = (j,s,m,p) \in \Lambda_s$, $s = \WQ{s(\lceil j_0 /2\rceil,q_0)} \in \mathbb{S}$ and $j \ge j_0$, let $\psi^{\sharp}_{\lambda} \in L^2(\R^2)$ be defined as in
Lemma \ref{lemm:technical_4}. Let $\hat{x}_2 \in \R$ so that $(E(\hat{x}_2),\hat{x}_2) \in \mathrm{supp}(\psi^{\sharp}_{\lambda})$ and
$\hat s = E^{'}(\hat{x}_2)$. Also let $k_s \in \Z$ so that \WQ{$s = \tfrac{k_s}{2^{\lceil j/2 \rceil}}$}. Then the following hold.
\begin{itemize}
\item[(i)] If $|\hat s| \leq 3$, then
\[|\langle f,\psi_{\lambda}\rangle| \lesssim \min\Bigl\{2^{-\frac{3}{4}j}, \frac{2^{-\frac{3}{4}j}2^{3\rho j}}{|k_s + 2^{\WQ{\lceil j/2 \rceil}}\hat{s}|^3}\Bigl\}.\]
\item[(ii)] If $|\hat{s}| > \frac{3}{2}$, then
\[
|\langle f,\psi_{\lambda}\rangle| \lesssim 2^{3\rho j}2^{-\frac{9}{4}j}.
\]
\end{itemize}
\end{lemma}

After these strategic discussions, we are now ready to present the proof of Theorem \ref{thm:sparsity}.

\begin{proof}[Proof of Theorem \ref{thm:sparsity}]
We start by defining dyadic cubes $Q_{j,\ell} \subseteq [0,1]^2$ for $j \ge 0$ and $\ell \in \Z^2$ by setting
\[
Q_{j,\ell} := \WQ{2^{-\lfloor j/2 \rfloor}}[0,1]^2+2^{\WQ{-\lfloor j/2 \rfloor}}\ell.
\]
The set of dyadic cubes intersecting the discontinuity curve $\Gamma$ of $f \in \mathcal{E}^2(\R^2)$ is then given by
\[
\mathcal{Q}_j = \{Q_{j,\ell} : \mathrm{int}(Q_{j,\ell}) \cap \Gamma \neq \emptyset\},
\]
where $\mathrm{int}(Q_{j,\ell})$ is the interior set of $Q_{j,\ell}$.

Next, without loss of generality, we assume that the discontinuity curve $\Gamma$ is given by $x_1 = E(x_2)$ with
$E \in C^2([0,1])$. In fact, for sufficiently large $j$, the discontinuity $\Gamma$ can be expressed as either $x_1 = E(x_2)$
or $x_2 = \tilde E(x_1)$ within $Q_{j,\ell} \in \mathcal{Q}_j$. Hence the same arguments can be applied for $x_2 = \tilde E(x_1)$
except for switching the order of variables. For each $Q_{j,\ell} \in \mathcal{Q}_j$, let now $E_{j,\ell}$ to be a $C^2$ function such that
\[
\Gamma \cap \mathrm{int}(Q_{j,\ell}) = \{(x_1,x_2) \in \mathrm{int}(Q_{j,\ell}) : x_1 = E_{j,\ell}(x_2)\}.
\]
This allows us to defined
\[
\mathcal{Q}^0_j := \{ Q_{j,\ell} \in \mathcal{Q}_j : \|E^{'}_{j,\ell}\|_{\infty} \leq 3 \}
\]
and
\[
\mathcal{Q}^1_j := \mathcal{Q}_j \cap (\mathcal{Q}^0_j)^c.
\]
Notice that, for all $E_{j,\ell}$ associated with $\mathcal{Q}_{j,\ell} \in Q^1_j$, we may assume
\begin{equation*}
\inf_{(x_1,x_2) \in \mathrm{int}(Q_{j,\ell})}|E^{'}_{j,\ell}(x_2)| > 3/2
\end{equation*}
for sufficiently large $j$.

We further define the orientation of the discontinuity curve $\Gamma$ in each dyadic cube $Q_{j,\ell}$ by
\begin{equation*}
\hat{s}_{j,\ell} = E^{'}_{j,\ell}(\hat{x}_2) \quad \text{for some} \,\, (E_{j,\ell}(\hat{x}_2),\hat{x}_2) \in \mathrm{int}(Q_{j,\ell}) \cap \Gamma.
\end{equation*}
Moreover, for any $J > 0$, we define $\mathbb{S}_{J/2}$ as a finite subset of $\mathbb{S}$ by
\[
\mathbb{S}_{J/2} = \{s(j,q) = 0 : j = 0, q = 0\} \cup \{s(\WQ{\lceil j/2\rceil},q) : 0 \leq j \leq J,\, |q| \leq 2^{\WQ{\lceil j/2 \rceil}}, \, q \in 2 \Z +1\}.
\]
Finally, let $k_{j,s} \in \Z$ be chosen so that
\[
s = \frac{k_{j,s}}{\WQ{2^{\lceil j/2 \rceil}}} \quad \mbox{for } s = \WQ{s(\lceil j_0 /2\rceil,q_0)}\in S_{J/2} \mbox{ and } j \ge j_0.
\]

We will now consider the two cases, namely when the shearlets intersect the discontinuity curve of $f$ or not
separately. For this, we define subsets $\Lambda^0$ and $\Lambda^1$ of the general index set $\Lambda =\{0,1\} \times \bigcup_{s \in \mathbb{S}} \Lambda_s$ by
\[
\Lambda^0 = \{\lambda \in \Lambda : \mathrm{int}(\supp(\psi^{\sharp}_{\lambda}))\cap \Gamma \neq \emptyset\} \quad \text{and} \quad \Lambda^1 = \Lambda \cap (\Lambda^0)^c.
\]
The smooth part, i.e., shearlet coefficients not intersecting the discontinuity curve $\Gamma$, can now be handled quickly.
Following the proof of
Proposition 2.1 in \cite{KL11}, for the approximated shearlet elements $\psi^{\sharp}_{\lambda}$ defined in Lemma \ref{lemm:technical_3}, one can show
that
\[
\sum_{\lambda \in (\WQ{\Lambda_N})^c \cap \Lambda^1} |\langle f,\psi^{\sharp}_{\lambda}\rangle|^2 \lesssim 2^{-J(1-\frac{13}{2}\rho)}
\]
with $N \sim J 2^{\frac{J(1+\rho)}{2}}$ as $J \rightarrow \infty$. By Lemma \ref{lemm:technical_3}(ii), this implies
\begin{equation}\label{eq:smooth}
\sum_{\lambda \in (\WQ{\Lambda_N})^c \cap \Lambda^1} |\langle f,\psi_{\lambda}\rangle|^2 \lesssim 2^{-J(1-\frac{13}{2}\rho)} \quad \mbox{as } J \to \infty.
\end{equation}

We now turn to analyze shearlets corresponding to $\Lambda^0$, aiming to prove that, again for $N \sim J 2^{\frac{J(1+\rho)}{2}}$,
\begin{equation}\label{eq:nonsmooth}
\sum_{\lambda \in (\WQ{\Lambda_N})^c \cap \Lambda^0} |\langle f,\psi_{\lambda}\rangle|^2 \lesssim 2^{-J(1-\frac{13}{2}\rho)} \quad \mbox{as } J \to \infty.
\end{equation}
For this, we fix some $J \ge 0$. Then we define subsets $\Lambda^0_j$, $j \ge 0$, of $\Lambda^0$ by
\begin{equation}\label{eq:index01}
\Lambda^0_j := \{\lambda = (j^{'},s,m,p)\in \Lambda^0 : j^{'} = j \}.
\end{equation}
Notice that $\Lambda^0 = \bigcup_{j = 0}^{\infty} \Lambda^0_j$. Further, for $r = 0,1$, define again subsets of
those sets by
\begin{eqnarray} \nonumber
\Lambda^{0}_{j,r} & := & \Bigl\{\lambda = (j,s,m,p) \in \Lambda^0_j : \mathrm{int}(\supp(\psi^{\sharp}_{\lambda})) \cap \mathrm{int}(Q_{j,\ell}) \cap \Gamma \neq \emptyset, \\ \label{eq:index02}
& & \hspace*{6cm}  \text{for} \,\, p \leq \max\Bigl(\tfrac{j\rho}{2},\tfrac{J\rho}{2}\Bigr), Q_{j,\ell} \in \mathcal{Q}^{r}_j\Bigr\}
\end{eqnarray}
corresponding to areas in which the discontinuity curve has a certain slope.

We further aim to collect all indices from the sets $\Lambda^0_j$ which correspond to significant shearlet coefficients. We might
overestimate at this point in the sense of also collecting indices corresponding to small shearlet coefficients; but it will turn
out in the end that this more or less crude collection is sufficient for deriving the anticipated sparse approximation behavior.
The first set for this purpose extracts such indices, which are related to the set \WQ{$\mathcal{Q}^0_j$}, by choosing
\[
\tilde{\Lambda}^0_j := \Bigl\{\lambda = (j,s,m,p) \in \Lambda^0_j : p \leq \tfrac{J\rho}{2} \Bigr\} \quad \mbox{for } j = 0,1,\dots,\WQ{\lceil J/4 \rceil}-1
\]
and
\begin{eqnarray*}
\tilde{\Lambda}^0_j &:=& \Bigl\{\lambda = (j,s,m,p) \in \Lambda^0_j : \mathrm{int}(\supp(\psi^{\sharp}_{\lambda})) \cap \mathrm{int}(Q_{j,\ell}) \cap \Gamma \neq \emptyset \\ \nonumber
&& \hspace*{2cm} \text{and } |k_{j,s} + \WQ{2^{\lceil j/2 \rceil}}\hat{s}_{j,\ell}| \leq 2^{\frac{J-j}{4}} \,\, \text{for} \,\, Q_{j,\ell} \in \WQ{\mathcal{Q}^0_j} \,\, \text{and} \,\, p \leq \tfrac{J\rho}{2}\Bigr\}
\quad \mbox{for } j \ge \WQ{\lceil J/4 \rceil}.
\end{eqnarray*}
Similar considerations lead to the following selection related to the set \WQ{$\mathcal{Q}^1_j$}:
\[
\tilde{\Lambda}^1_j := \Bigl\{\lambda = (j,s,m,p) \in \Lambda^0_j : \mathrm{int}(\supp(\psi^{\sharp}_{\lambda})) \cap \mathrm{int}(Q_{j,\ell}) \cap \Gamma \neq \emptyset \,\,
\text{for} \,\, Q_{j,\ell} \in \mathcal{Q}^1_j \,\, \text{and} \,\, p \leq \tfrac{J\rho}{2}\Bigr\}.
\]
Finally, we define $\tilde{\Lambda}^0$ as a set of indices $\lambda \in \Lambda$ containing all significant shearlet coefficients $\langle f,\psi_{\lambda}\rangle$
(and presumably also others) as follows:
\begin{equation}\label{eq:subindex}
\tilde{\Lambda}^0 := \Bigl( \bigcup_{j = 0}^{J} \tilde{\Lambda}^0_j\Bigr) \bigcup \Bigl(\WQ{\bigcup_{j = \lceil J/4 \rceil}^{\lceil J/3 \rceil}} \tilde{\Lambda}^1_j\Bigr).
\end{equation}

We now turn to estimate $\#(\tilde{\Lambda}^0)$. Using the same argument as in \cite{KL11}(page 19), for $s \in \mathbb{S}$, $j \ge 0$ fixed
and each $Q_{j,\ell} \in \mathcal{Q}^0_j$, we obtain
\begin{eqnarray}\nonumber
\lefteqn{\hspace*{-1cm}\#\bigl( \Bigl\{\lambda = (j,s,m,p) \in \WQ{\Lambda^0_j} : \mathrm{int}(\supp(\psi^{\sharp}_{\lambda})) \cap \mathrm{int}(Q_{j,\ell}) \cap \Gamma \neq \emptyset \,\,
\text{for}\,\, p \leq\tfrac{J\rho}{2}\Bigr\}  \bigr)}\\ \label{eq:count01}
& \hspace*{9cm} \lesssim & 2^{\frac{J\rho}{2}}(1+|\hat{k}_{j,\ell}(s)|),
\end{eqnarray}
where $\hat{k}_{j,\ell}(s) = k_{j,s} + \WQ{2^{\lceil j/2 \rceil}}\hat{s}_{j,\ell}$ and the additional factor $2^{\frac{J\rho}{2}}$ comes from
oversampling parameter $p$ associated with the sampling matrix $D_p$ in \eqref{eq:formclassical} for $\lambda = (j,s,m,p)$.
Also, for $p \in \N_0$ and $j \ge 0$ fixed, it is immediate that
\begin{equation}\label{eq:count02}
\#(\{\lambda = (j^{'},s,m,p^{'}) \in \Lambda^0_j : j^{'} = j \,\,\text{and}\,\,p^{'} = p\}) \lesssim 2^{2j}2^p \quad \mbox{for } j = 0,\dots,\WQ{\lceil J/4 \rceil}-1.
\end{equation}
Finally, we obtain
\begin{equation}\label{eq:count03}
\#(\{\lambda = (j^{'},s,m,p^{'}) \in \Lambda^0_{j,r} : j^{'} = j \,\,\text{and}\,\,p^{'} = p\}) \lesssim 2^{\frac{3}{2}j}2^p  \quad \mbox{for } j \ge 0, \; r=0,1,
\end{equation}
from arguing as follows: There exist at most about $2^{j+p}$ shearlets $\psi_{\lambda}$ whose approximated part $\psi^{\sharp}_{\lambda}$ intersects $\Gamma$
for $\lambda = (j,s,m,p)$ with fixed $j$, $s$, and $p$. Also, if $\lambda = (j,s,m,p) \in \Lambda^0_j$, then $s \in \mathbb{S}_{j/2}$ and
$\#(\mathbb{S}_{j/2}) \lesssim 2^{j/2}$. Thus, in this case, there are about $2^{j+p}$ translates with respect to $m$ and $2^{j/2}$ shearings with respect
to $s$ yielding the estimate in \eqref{eq:count03}.

By \eqref{eq:count01}--\eqref{eq:count03}, we now derive an estimate for $\#(\tilde{\Lambda}^0)$ as follows:
\allowdisplaybreaks{
\begin{eqnarray}\label{eq:counting}
\#(\tilde{\Lambda}^0) &\lesssim& \sum_{j = 0}^{\WQ{\lceil J/4 \rceil}-1} \#(\tilde{\Lambda}^0_j) +
\sum_{j = \WQ{\lceil J/4 \rceil}}^{\WQ{\lceil J/3 \rceil}} \#(\tilde{\Lambda}^1_j) +
\sum_{j = \WQ{\lceil J/4 \rceil}}^{J} \#(\tilde{\Lambda}^0_j) \nonumber \\
&\lesssim& \sum_{j = 0}^{\WQ{\lceil J/4 \rceil}-1}(2^{2j}2^{\frac{J\rho}{2}})+ \sum_{j = \WQ{\lceil J/4 \rceil}}^{\WQ{\lceil J/3 \rceil}} (2^{\frac{3}{2}j})2^{\frac{J\rho}{2}}  \nonumber \\
& & + \sum_{j = \WQ{\lceil J/4 \rceil}}^{J}\sum_{\{\ell : Q_{j,\ell} \in \mathcal{Q}^0_j\}}\sum_{\{s \in \mathbb{S}_{j/2} : |\hat{k}_{j,\ell}(s)| \leq 2^{\frac{J-j}{4}}\}}
2^{\frac{J\rho}{2}} (1+(1+|\hat{k}_{j,\ell}(s)|) \nonumber \\
&\lesssim& 2^{\frac{J}{2}(1+\rho)} + 2^{\frac{J\rho}{2}}\sum_{j = J/4}^{J} \#(\mathcal{Q}^0_j)(2^{\frac{J-j}{2}}) \nonumber \\
&\lesssim& J2^{\frac{J}{2}(1+\rho)}.
\end{eqnarray}
}

Let now $N >0$ be given. Then we choose $J > 0$ such that $N \sim J2^{\frac{J}{2}(1+\rho)}$. Without loss of generality, we may assume
that $N \ge \#(\tilde{\Lambda}^0)$. Then we have
\begin{eqnarray*}
\sum_{\lambda \in \Lambda^0 \cap (\Lambda_N)^c} |\langle f,\psi_{\lambda} \rangle|^2 &\leq& \sum_{\lambda \in \Lambda^0 \cap (\tilde{\Lambda}^0)^c} |\langle f,\psi_{\lambda}\rangle|^2 \\
&\lesssim& \sum_{j = \WQ{\lceil J/4 \rceil}}^{J} \sum_{\lambda \in \Lambda^0_{j,0} \cap (\tilde{\Lambda}^0_j)^c} |\langle f,\psi_{\lambda}\rangle|^2
+ \sum_{j = \WQ{\lceil J/4 \rceil}}^{\WQ{\lceil J/3 \rceil}} \sum_{\lambda \in \Lambda^0_{j,1} \cap (\tilde{\Lambda}^1_j)^c} |\langle f,\psi_{\lambda}\rangle|^2 \\
& & + \sum_{j = \WQ{\lceil J/3 \rceil}+1}^{J} \sum_{\lambda \in \Lambda^0_{j,1}} |\langle f,\psi_{\lambda}\rangle|^2
+ \sum_{j = 0}^{J} \sum_{\lambda \in \Lambda^0_j \cap (\Lambda^0_{j,\WQ{0}} \cup \Lambda^0_{j,1})^c} |\langle f,\psi_{\lambda}\rangle|^2 \\
& & + \sum_{j \ge J} \sum_{\lambda \in \Lambda^0_j} |\langle f,\psi_{\lambda}\rangle|^2 \\
&=& \mathrm{(I)} + \mathrm{(II)} + \mathrm{(III)} + \mathrm{(IV)} + \mathrm{(V)}.
\end{eqnarray*}
For the second inequality above, we used the fact that $\tilde{\Lambda}^0_j = \Lambda^0_{j,0} \cup \Lambda^0_{j,1}$ for $j < \WQ{\lceil J/4 \rceil}$.

We now estimate (I) -- (V). For this, for each $s \in \mathbb{S}_{j/2}$ and $Q_{j,\ell} \in \mathcal{Q}^0_j$, let
\[
\hat{k}_{j,\ell}(s) = k_{j,s} + 2^{j/2}\hat{s}_{j,\ell}.
\]
We start with (I). Using Lemma \ref{lemm:technical_5}(i) and \eqref{eq:count01}, we obtain
\begin{eqnarray}\label{eq:estimate01}
\mathrm{(I)} &\lesssim& \sum_{j = \WQ{\lceil J/4 \rceil}}^{J} \sum_{\{\ell : Q_{j,\ell} \in \mathcal{Q}^0_j\}}
\sum_{\{s \in \mathbb{S}_{j/2} : |\hat{k}_{j\ell}(s)| > 2^{\frac{J-j}{4}}\}} 2^{\frac{J\rho}{2}}(1+|\hat{k}_{j,\ell}(s)|)
\Biggl(\frac{2^{3\rho j}2^{-\frac{3}{4}j}}{|\hat{k}_{j,\ell}(s)|^3}\Biggr)^2 \nonumber \\
&\lesssim&  \sum_{j = \WQ{\lceil J/4 \rceil}}^{J} \sum_{\{\ell : Q_{j,\ell} \in \mathcal{Q}^0_j\}} 2^{\frac{J\rho}{2}}2^{6\rho j}2^{-\frac{3}{2}j}(2^{-\frac{J-j}{4}})^{4} \nonumber \\
&\lesssim& \sum_{j = \WQ{\lceil J/4 \rceil}}^{J} (2^{j/2})2^{\frac{J\rho}{2}}2^{6\rho j}2^{-\frac{3}{2}j}(2^{-\frac{J-j}{4}})^{4} \lesssim 2^{-J(1-\frac{13}{2}\rho)}.
\end{eqnarray}

Second we turn to (II). For this, we notice that $\Lambda^0_{j,1} = \tilde{\Lambda}^1_j$ for $j \leq \WQ{\lceil J/3 \rceil}$.
But this immediately implies $\mathrm{(II)} = 0$.

To estimate (III), we use Lemma \ref{lemm:technical_5}(ii) and \eqref{eq:count03} to obtain
\beq \label{eq:estimate03}
\mathrm{(III)} \lesssim \sum_{j = \WQ{\lceil J/3 \rceil}}^{J} 2^{\frac{J\rho}{2}}(2^{\frac{3}{2}j})(2^{3\rho j}2^{-\frac{9}{4}j})^2
\lesssim \sum_{j = \WQ{\lceil J/3 \rceil}}^{J} 2^{\frac{J\rho}{2}}2^{-3j(1-2\rho)}
\lesssim 2^{\frac{J\rho}{2}}2^{-J(1-2\rho)} = 2^{-J(1-5/2\rho)}.
\eeq
For the third inequality, we used that $\rho < \frac{1}{2}$.

Term (IV) is estimated by using Lemma \ref{lemm:technical_2}(ii) and \eqref{eq:count02}. Using also $\alpha \ge \frac{6}{\rho}+1$,
we have
\begin{eqnarray}\label{eq:estimate04}
\mathrm{(IV)} &\lesssim& \sum_{j = 0}^{J}\sum_{p = 0}^{\infty}(2^{\frac{J\rho}{2}+p})(2^{2j})(2^{-\alpha(\frac{J}{4}\rho+\frac{p}{2})})^2 \nonumber \\
&\lesssim& \sum_{j = 0}^{J} \Bigl(\sum_{p = 0}^{\infty} 2^{-p(\alpha-1)}\Bigr) 2^{\frac{J\rho}{2}}2^{2j}2^{-\alpha\frac{J\rho}{2}} \nonumber \\
&\lesssim& \sum_{j = 0}^{J} 2^{2J-\alpha\frac{J\rho}{2}+\frac{J\rho}{2}} \lesssim 2^{-J}.
\end{eqnarray}

Finally, the last term can be estimated by
\begin{eqnarray}\nonumber
\mathrm{(V)} &\leq& \sum_{j \ge J} \sum_{\lambda \in \Lambda^0_{j,0}} |\langle f,\psi_{\lambda}\rangle|^2 + \sum_{j \ge J} \sum_{\lambda \in \Lambda^0_{j,1}}
|\langle f,\psi_{\lambda}\rangle|^2 + \sum_{j \ge J} \sum_{\lambda \in \Lambda_j \cap (\Lambda^0_{j,0} \cup \Lambda^0_{j,1})^c} |\langle f,\psi_{\lambda}\rangle|^2 \\ \label{eq:estimate10}
&=& \mathrm{(A)} + \mathrm{(B)} + \mathrm{(C)}.
\end{eqnarray}
It remains to analyze the terms (A) -- (C). We start with (A). By Lemmata \ref{lemm:technical_2}(i) and \ref{lemm:technical_5}(i) we well
as \eqref{eq:count01}, we obtain
\begin{eqnarray*}
\mathrm{(A)} &\lesssim& \sum_{j \ge J} \sum_{\{\ell : Q_{j,\ell} \in \mathcal{Q}^0_j\}}
2^{\frac{j\rho}{2}}\sum_{s \in \mathbb{S}_{j/2}}(1+|\hat{k}_{j,\ell}(s)|)\min\Bigl\{2^{-\frac{3}{4}j},\frac{2^{3\rho j}2^{-\frac{3}{4}j}}{|\hat{k}_{j,\ell}(s)|^3}\Bigr\}^2 \nonumber \\
&\lesssim&  \sum_{j \ge J} 2^{\frac{j\rho}{2}}(\#(\mathcal{Q}^0_j)2^{-\frac{3}{2}j}2^{6\rho j})  \nonumber \\
&\lesssim& \sum_{j \ge J} 2^{-j}2^{\frac{13 j\rho}{2}} \lesssim 2^{-J(1-\frac{13\rho}{2})}.
\end{eqnarray*}
The terms (B) and (C) can be estimated by using Lemmata \ref{lemm:technical_5}(ii) and \ref{lemm:technical_2}(ii) as well as equations \eqref{eq:count02} and
\eqref{eq:count03} to obtain
\[
\mathrm{(B)} \lesssim \sum_{j \ge J} (2^{\frac{3}{2}j})(2^{\frac{j\rho}{2}})(2^{-\frac{9}{4}j}2^{3\rho j})^2
= \sum_{j \ge J} 2^{-3j}2^{6\rho j}2^{\frac{j\rho}{2}} \lesssim 2^{-3J+\frac{13}{2}\rho J}.
\]
and
\[
\mathrm{(C)} \lesssim \sum_{j \ge J}\sum_{p = 0}^{\infty}(2^{\frac{j\rho}{2}+p})(2^{2j})(2^{-\alpha(j\frac{\rho}{4}+\frac{p}{2})})^2
= \sum_{j \ge J} \Bigl( \sum_{p = 0}^{\infty} 2^{-p(\alpha-1)}\Bigr) 2^{2j+\frac{j\rho}{2}-j\frac{\alpha\rho}{2}} \lesssim 2^{-J}.
\]
Thus, continuing \eqref{eq:estimate10},
\begin{equation}\label{eq:estimate05}
\mathrm{(V)} \lesssim 2^{-J(1-\frac{13\rho}{2})}.
\end{equation}
Summarizing, \eqref{eq:estimate01}---\eqref{eq:estimate05} imply \eqref{eq:nonsmooth}.

Finally, using \eqref{eq:smooth}, \eqref{eq:nonsmooth}, and the frame property of the shearlet system $\mathcal{SH}(\varphi_1,\psi_1; g)$,
we can conclude that
\[
\|f-f_N\|^2_2 \lesssim \sum_{\lambda \notin \Lambda_N} |\langle f,\psi_{\lambda}\rangle|^2 \lesssim 2^{-J(1-\frac{13\rho}{2})}
\]
with  $N \sim J2^{\frac{J(1+\rho)}{2}}$, which implies our claim.
\end{proof}

\section{Proofs of Preliminary Lemmata and Corollary \ref{cor:approx}}\label{sec:proofs}

\subsection{Proof of Lemma \ref{lemm:technical_1}}\label{subsec:technical_1}

We first observe that
\[
(\widehat{Q_{j-j_0} * \psi^p})(\xi) = \sum_{j^{'}=j_0-j}^{\infty} \hat{h}_{j',p}(\xi), \quad \mbox{where} \,\, \hat{h}_{j',p}(\xi) := |\hat g(A^{-1}_{j^{'}}\xi)|^2\hat{\psi}^p(\xi),
\]
with
\[
\mathrm{supp}(h_{j',p}) \subset A_{\max\{-j^{'},0\}}[-C,C]^2
\]
for some $C>0$.
We can then estimate $\|Q_{j-j_0} * \psi^p\|_1$ by
\begin{eqnarray*}
\|Q_{j-j_0} * \psi^p\|_1 &=& \int_{\R^2} \Bigl| \int_{\R^2} (\widehat{Q_{j-j_0} * \psi^p})(\xi)e^{2\pi i \ip{\xi}{x}} d\xi\Bigr|dx \\
&\lesssim& \sum_{j^{'} = 0}^{\infty}\int_{[-C,C]^2}\int_{\R^2}|\hat{h}_{j^{'},p}(\xi)|d\xi dx + \sum_{j^{'} = 0}^{j-j_0}\int_{A_{j^{'}}[-C,C]^2}\int_{\R^2} |\hat{h}_{-j^{'},p}(\xi)|d\xi dx \\
&=& \mathrm{(I)} + \mathrm{(II)}.
\end{eqnarray*}
We will use the following inequality to estimate (I) and (II).
Assume that $j_2 \ge j_1$ and $\beta > \alpha +1$ with $\alpha>0$.
Then
\begin{equation}\label{eq:inequality}
\int_{\R}\frac{2^{-j_1}\min\{1,|2^{-j_2}x|\}^{\alpha}}{(1+|2^{-j_1}x|)^{\beta}}dx \lesssim 2^{-\alpha(j_2-j_1)}.
\end{equation}
We are now ready to estimate (I) and (II). First, by \eqref{eq:decay_varphi_1}-\eqref{eq:decay_g}, we have
\begin{eqnarray*}
\mathrm{(I)} &\lesssim& \sum_{j = 0}^{\infty} \int_{[-C,C]^2}\int_{\R^2} \frac{\min\{1,|2^{-j}\xi_1|\}^{\alpha}\min\{1,|\xi_1|\}^{\alpha}
\min\{1,|2^{-p}\xi_2|\}^{\alpha}}{(1+|2^{-j}\xi_1|)^{\beta}(1+|2^{-j/2}\xi_2|)^{\beta}(1+|\xi_1|)^{\beta}(1+|2^{-p}\xi_2|)^{\beta}}d\xi dx \\
&\lesssim& \sum_{j = 0}^{2p-1} \int_{[-C,C]^2} \int_{\R}\frac{\min\{1,|2^{-j}\xi_1|\}^{\alpha}}{(1+|\xi_1|)^{\beta}}d\xi_1
\int_{\R} \frac{\min\{1,|2^{-p}\xi_2|\}^{\alpha}}{(1+|2^{-j/2}\xi_2|)^{\beta}}d\xi_2 dx \\
& & + \sum_{j = 2p}^{\infty}\int_{[-C,C]^2}\int_{\R} \frac{\min\{1,|2^{-j}\xi_1|\}^{\alpha}}{(1+|\xi_1|)^{\beta}} d\xi_1
\int_{\R}\frac{\min\{1,|2^{-p}\xi_2|\}^{\alpha}}{(1+|2^{-p}\xi_2|)^{\beta}}d\xi_2 dx \\
&\lesssim& \sum_{j = 0}^{2p-1} 2^{-j(\alpha-1)}2^{-\alpha(p-j/2)}2^{p} + \sum_{j = 2p}^{\infty}2^{-j(\alpha-1)}2^{p} \lesssim 2^{-p(\alpha-1)} + 2^{-p(2\alpha-3)} \leq 2.
\end{eqnarray*}
Second,
\allowdisplaybreaks{
\begin{eqnarray*}
\mathrm{(II)} &\lesssim& \sum_{j = 0}^{\infty}\int_{A_j[-C,C]^2}\int_{\R^2}
\frac{\min\{1,|2^j\xi_1|\}^{\alpha}\min\{1,|\xi_1|\}^{\alpha}\min\{1,|2^{-p}\xi_2|\}^{\alpha}}
{(1+|2^{j/2}\xi_2|)^{\beta}(1+|2^{j}\xi_1|)^{\beta}(1+|2^{-p}\xi_2|)^{\beta}}d\xi dx \\
&\lesssim& \sum_{j = 0}^{\infty} \int_{A_j[-C,C]^2}\int_{\R} \frac{\min\{1,|\xi_1|\}^{\alpha}}{(1+|2^{j}\xi_1|)^{\beta}}d\xi_1
\int_{\R} \frac{\min\{1,|2^{-p}\xi_2|\}^{\alpha}}{(1+|2^{j/2}\xi_2|)^{\beta}(1+|2^{-p}\xi_2|)^{\beta}} d\xi_2 dx \\
&\lesssim& \sum_{j = 0}^{\infty} \int_{A_j[-C,C]^2} 2^{-j}2^{-\alpha j}
\int_{\R} \frac{\min\{1,|2^{-p}\xi_2|\}^{2}}{(1+|2^{j/2}\xi_2|)^{2}(1+|2^{-p}\xi_2|)^{2}}
d\xi_2 dx \\
&\lesssim&
\sum_{j = 0}^{\infty} \int_{A_j[-C,C]^2} 2^{-j}2^{-\alpha j}
\int_{\R} \frac{2^{-j/2}}{(1+|\xi_2|)^{2}}
d\xi_2 dx \\
&\lesssim& \sum_{j = 0}^{\infty} 2^{\frac{3}{2} \cdot j} 2^{-j} 2^{-\alpha j} 2^{-j/2} \leq 2.
\end{eqnarray*}
}
Therefore, (I) and (II) are uniformly bounded, which implies the uniform boundedness of $\|Q_{j-j_0}*\psi^p\|_1$.

\subsection{Proof of Lemma \ref{lemm:technical_2}}\label{subsec:technical_2}

We start by proving (i). First, note that Proposition \ref{prop:comparison}(ii) implies the form
\begin{equation*}
\psi_{\lambda} = |\mathrm{det}(A_j)|^{1/2}(\Theta_{j-j_0}*\psi^p)(S_kA_j\cdot - \WQ{D_p}m),
\end{equation*}
where $\lambda = (j,s,m,p) \in \Lambda$ with $s = \WQ{s(\lceil j_0 /2\rceil,q_0)} \in \mathbb{S}$, $j \ge j_0$ and $k \in \Z$ with \WQ{$s = \frac{k}{2^{\lceil j/2 \rceil}}$}.
This allows us to estimate $|\langle f,\psi_{\lambda}\rangle|$ as follows:
\begin{eqnarray*}
|\langle f,\psi_{\lambda}\rangle| &\leq& |\mathrm{det}(A_j)|^{1/2}\int_{\R^2} |(\Theta_{j-j_0}*\psi^p)(S_kA_j x - \WQ{D_p}m)f(x)| dx \nonumber \\
&\leq& 2^{-\frac{3}{4}j} \int_{\R^2} |f(A^{-1}_jS^{-1}_k(y+\WQ{D_p}m))||(\Theta_{j-j_0}*\psi^p)(y)|dy \nonumber \\
&\leq& 2^{-\frac{3}{4}j}\|f\|_{\infty} \|\Theta_{j-j_0}*\psi^p\|_1.
\end{eqnarray*}
The claim in (i) now follows from Lemma \ref{lemm:technical_1}.

We next turn to proving (ii). Since by definition, $\psi_{\lambda} = G_s * \psi_{j,s,m,p}$, we have
\[
|\langle f,\psi_{\lambda}\rangle|^2 \leq \|\hat f\|^2_2 \|\hat{G}_s\cdot\hat{\psi}_{j,s,m,p}\|_2^2.
\]
Now we estimate $\|\hat{G}_s\cdot\hat{\psi}_{j,s,m,p}\|_2^2$ as follows.  By Lemma \ref{lemm:framePhiPsi},
\begin{eqnarray*}
\|\hat{G}_s\cdot\hat{\psi}_{j,s,m,p}\|_2^2 &\lesssim& \int_{\R^2}\Bigl|\sum_{j^{'}=0}^{\infty} |\hat{g}(A_{-j^{'}}\xi)|^2\Bigr|^2|\hat{\psi}_{j,0,m,p}(\xi)|^2d\xi \\
&\lesssim& \int_{\R^2}\sum_{j^{'}=0}^{\infty} |\hat{g}(A_{-j^{'}}\xi)|^2|\hat{\psi}_{j,0,m,p}(\xi)|^2d\xi
\\
&\lesssim& \sum_{j^{'} = 0}^{j+p}\int_{\R^2}\cdots + \sum_{j^{'} = j+p}^{\infty}\int_{\R^2}\cdots = \mathrm{(I)} + \mathrm{(II)}.
\end{eqnarray*}
Note that \eqref{eq:decay_varphi_1} and \eqref{eq:decay_g} imply that
\begin{equation}\label{decay01}
|\hat{\psi}_{j,0,m,p}(\xi)|^2 \lesssim 2^{-\frac{3}{2}j-p}\cdot\frac{\min\{1,|2^{-j}\xi_1|^{2\alpha}\}}{(1+|2^{-j}\xi_1|)^{2\beta}}
\cdot\frac{\min\{1,|2^{-j/2-p}\xi_2|^{2\alpha}\}}{(1+|2^{-j/2-p}\xi_2|)^{2\beta}}.
\end{equation}
Hence, using \eqref{eq:decay_g}, \eqref{eq:inequality} and \eqref{decay01}, can estimate (I) by
\begin{eqnarray*}
\mathrm{(I)} &\lesssim&  \sum_{j^{'}=0}^{j+p} 2^{-j/2-p}
\int_{\R}\frac{\min\{1,|2^{-j/2-p}\xi_2|^{2\alpha}\}}{(1+|2^{-j^{'}/2}\xi_2|)^{2\beta}}d\xi_2\int_{\R}\frac{2^{-j}}{(1+|2^{-j}\xi_1|)^{2\beta}}d\xi_1\\
&\lesssim& \sum_{j^{'} = 0}^{j+p}2^{-j/2-p}\int_{\R}
\frac{\min\{1,|2^{-j/2-p}\xi_2|^{2\alpha}\}}{(1+|2^{-j^{'}/2}\xi_2|)^{2\beta}}d\xi_2 \\
&\lesssim& 2^{-\alpha p}.
\end{eqnarray*}
Similarly, we can estimate (II) as follows:
{\allowdisplaybreaks
\begin{eqnarray*}
\mathrm{(II)} &\lesssim& \sum_{j^{'}=j+p}^{\infty}2^{-j}\int_{\R}\frac{\min\{1,|2^{-j^{'}}\xi_1|^{2\alpha}\}}
{(1+|2^{-j}\xi_1|)^{2\beta}}\WQ{d\xi_1}\int_{\R}\frac{2^{-j/2-p}}{(1+|2^{-j/2-p}\xi_2|)^{2\beta}}d\xi_2 \\
&\lesssim& \sum_{j^{'}=j+p}^{\infty} 2^{-j}\int_{\R}\frac{\min\{1,|2^{-j^{'}}\xi_1|^{2\alpha}\}}{(1+|2^{-j}\xi_1|)^{2\beta}}d\xi_1 \\
&\lesssim& 2^{-2\alpha p}.
\end{eqnarray*}
}
This proves (ii).

\subsection{Proof of Lemma \ref{lemm:technical_3}}\label{subsec:technical_3}

First, note that (i) is obvious from the definition of $\psi^{\sharp}_{\lambda}$ and
\[
|\langle f,(\psi_{\lambda}-\psi^{\sharp}_{\lambda})\rangle|^2 \leq
\|f\|^2_2 \|\psi_{\lambda}-\psi^{\sharp}_{\lambda}\|^2_2.
\]
Next, we estimate $\|\psi_{\lambda}-\psi^{\sharp}_{\lambda}\|^2_2$ to show (ii). By using \eqref{eq:decay_varphi_1},
\eqref{eq:decay_g}, and \eqref{eq:inequality}, we obtain
\begin{eqnarray*}
\|\psi_{\lambda}-\psi^{\sharp}_{\lambda}\|^2_2 &\lesssim& \int_{\R^2}
\Bigl| \sum_{j^{'} = 0}^{\lfloor j(1-\rho) \rfloor}|\hat{g}(A^{-1}_{j^{'}}\xi)|^2|\hat{\psi}_{j,0,0,p}(\xi)|\Bigr|^2 d\xi \\
&\lesssim& \int_{\R^2}
\sum_{j^{'} = 0}^{\lfloor j(1-\rho) \rfloor}|\hat{g}(A^{-1}_{j^{'}}\xi)|^2|\hat{\psi}_{j,0,0,p}(\xi)|^2 d\xi \\
&\lesssim& \sum_{j^{'} = 0}^{\lfloor j(1-\rho) \rfloor}\int_{\R^2}\frac{2^{-j^{'}}\min\{1,|2^{-j}\xi_1|\}^{\alpha}}{(1+|2^{-j^{'}}\xi_1|)^{\beta}}\cdot
\frac{2^{-p-j/2}}{(1+|2^{-j/2-p}\xi_2|)^{2\beta}} d\xi \\
&\lesssim& \sum_{j^{'} = 0}^{\lfloor j(1-\rho) \rfloor}2^{-(j-j^{'})\alpha} \lesssim 2^{-j\rho\alpha}.
\end{eqnarray*}
This proves our claim.

\subsection{Proof of Lemma \ref{lemm:technical_4}}\label{subsec:technical_4}

We only consider case $\ell = 1$, since the other case can be shown similarly.
By \eqref{eq:decay_varphi_1} and  \eqref{eq:decay_g}, we have
\begin{eqnarray*}
\lefteqn{\Bigl|\Bigl(\frac{\partial}{\partial \xi_2}\Bigr)\widehat{\Theta_{j-\WQ{\max\{\lfloor j(1-\rho) \rfloor,j_0\}}}*\psi^p}(\xi)\Bigr|}\\
&\lesssim& \sum_{j = 0}^{\infty} \Bigl|\Bigl(\frac{\partial}{\partial \xi_2}\Bigr)|\hat g(A_j\xi)|^2\hat\psi^p(\xi)\Bigr|
+ \sum_{j=0}^{\infty}\Bigl||\hat g(A_j\xi)|^2\Bigl(\frac{\partial}{\partial \xi_2}\Bigr)\hat \psi^p(\xi)\Bigr| \\
& & + \sum_{j = 1}^{\infty} \Bigl|\Bigl(\frac{\partial}{\partial \xi_2}\Bigr)|\hat g(A^{-1}_j\xi)|^2\hat\psi^p(\xi)\Bigr|
+ \sum_{j=1}^{\infty}\Bigl||\hat{g}(A^{-1}_j\xi)|^2\Bigl(\frac{\partial}{\partial \xi_2}\Bigr)\hat \psi^p(\xi)\Bigr| \\
&\lesssim& \sum_{j=0}^{\infty}2^{j/2}\frac{\min\{1,|2^{j}\xi_1|^{\alpha}\}\min\{1,|\xi_1|^{\alpha}\}\min\{1,|2^{-p}\xi_2|^{\alpha}\}}{(1+|2^{j}\xi_1|)^{\beta}
(1+|2^{j/2}\xi_2|)^{\beta}(1+|\xi_1|)^{\beta}(1+|2^{-p}\xi_2|)^{\beta}} \\
& & + \sum_{j=1}^{\infty}\frac{\min\{1,|2^{-j}\xi_1|^{\alpha}\}\min\{1,|\xi_1|^{\alpha}\}\min\{1,|2^{-p}\xi_2|^{\alpha}\}}{(1+|2^{-j}\xi_1|)^{\beta}
(1+|2^{-j/2}\xi_2|)^{\beta}(1+|\xi_1|)^{\beta}(1+|2^{-p}\xi_2|)^{\beta}} \\ &=& \mathrm{(I)} + \mathrm{(II)}
\end{eqnarray*}
Then we have
{\allowdisplaybreaks
\begin{eqnarray*}
\mathrm{(I)} &\lesssim& \sum_{j = 0}^{\infty} \frac{2^{-j/2}|\xi_1|\min\{1,|2^{j}|\xi_1|^{\alpha}\}\min\{1,|\xi_1|^{\alpha-1}\}
\min\{1,|2^{-p}\xi_2|^{\alpha}\}}{|\xi_1|(1+|2^j\xi_2|)^{\beta-1}(1+|2^{ j/2}\xi_2|)^{\beta}(1+|\xi_1|)^{\beta}(1+|2^{-p}\xi_2|)^{\beta}} \\
&\lesssim& \frac{\min\{1,|\xi_1|^{\alpha-1}\}}{(1+|\xi_2|)^{\beta-1}(1+|\xi_1|)^{\beta}}
\end{eqnarray*}
}
and
\begin{eqnarray*}
\mathrm{(II)} &\lesssim& \sum_{j = 1}^{\infty} \frac{\min\{1,|\xi_1|^{\alpha}\}2^{-p\gamma}|\xi_2|^{\gamma}2^{\frac{j}{2}\gamma}2^{-j\gamma}|\xi_1|^{\gamma}}{(1+|\xi_1|)^{\beta}
(1+|2^{-p}\xi_2|)^{\gamma}|\xi_2|^{\gamma}} \\
&\lesssim& \frac{\min\{1,|\xi_1|^{\alpha}\}}{(1+|\xi_1|)^{\beta-\gamma}(1+|\xi_2|)^{\gamma}}.
\end{eqnarray*}
These estimates for (I) and (II) then prove the lemma.

\subsection{Proof of Lemma \ref{lemm:technical_5}}\label{subsec:technical_5}
For some $q \in L^2(\R^2)$, set
\[
q_{j,k_s,{m}} := |\mathrm{det}(A_j)|^{1/2}q(S_{k_s}A_j\cdot - {m})
\]
for $j \ge 0$, $k_s \in \Z$ and $m \in \Z^2$. Further, assume that
\begin{equation*}
\supp(q_{j,k_s,m}) \subset A^{-1}_jS^{-1}_{k_s}\Bigl([-2^{\rho j}L,2^{\rho j}L] \times [-2^{\rho j/2}L,2^{\rho j/2}L]+m\Bigr)
\end{equation*}
for  some $L>0$. Provided that in addition, for $\alpha_1 \ge 5$, $\alpha_2 \ge 4$, and $h \in L^1(\R)$, we have
\[
|\hat{q}(\xi)| \lesssim \frac{\min\{1,|\xi_1|^{\alpha_1}\}}{(1+|\xi_1|)^{\alpha_2}(1+|\xi_2|)^{\alpha_2}}
\quad \mbox{and} \quad
\Bigl|\frac{\partial}{\partial \xi_2}\hat q(\xi) \Bigr| \leq |h(\xi_1)|\cdot \Bigl(1+\frac{|\xi_2|}{|\xi_1|}\Bigr)^{-\alpha_2},
\]
by following the proof of Proposition 2.2 in \cite{KL11}, we can show that
\begin{equation}\label{eq:est01}
|\langle f,q_{j,k_s,m}\rangle| \lesssim \frac{2^{-\frac{3}{4}j}2^{3\rho j}}{|k_s+2^{\WQ{\lceil j/2 \rceil}}\hat{s}|^3} \quad \text{if} \,\, |\hat s| \leq 3
\end{equation}
and
\begin{equation}\label{eq:est02}
|\langle f,q_{j,k_s,m}\rangle| \lesssim 2^{3\rho j}2^{-\frac{9}{4}}
\quad \text{if} \,\, |\hat s| > 3/2.
\end{equation}

We next choose $q_{j,k_s,{m}} := \psi^{\sharp}_{\lambda}$ for $\lambda = (j,s,m,p) \in \Lambda_s$ with $s = \WQ{s(\lceil j_0 /2\rceil,q_0)} \in \mathbb{S}$,
$s = \frac{k_s}{2^{\WQ{\lceil j/2 \rceil}}}$, and $j \ge j_0$. Hence, in particular, $q = \Theta_{j-\max\{\lfloor j(1-\rho) \rfloor,j_0\}}*\psi^p$. By Lemma \ref{lemm:technical_3}(i)
and Lemma \ref{lemm:technical_4}, we derive \eqref{eq:est01} and \eqref{eq:est02} for this choice. Thus, by Lemma \ref{lemm:technical_3}(ii),
\begin{equation}\label{eq:eestimate01}
|\langle f,\psi_{\lambda}\rangle| \lesssim \frac{2^{-\frac{3}{4}j}2^{3\rho j}}{|k_s+2^{\WQ{\lceil j/2 \rceil}}\hat{s}|^3} + 2^{-\frac{j}{2}\rho \alpha} \quad \text{if} \,\, |\hat s| \leq 3
\end{equation}
and
\begin{equation}\label{eq:eestimate02}
|\langle f,\psi_{\lambda}\rangle| \lesssim 2^{3\rho j}2^{-\frac{9}{4}j} + 2^{-\frac{j}{2}\rho \alpha} \quad \text{if} \,\, |\hat s| > 3/2.
\end{equation}
Finally, Lemma \ref{lemm:technical_2}(i), \eqref{eq:eestimate01}, and \eqref{eq:eestimate02} imply (i) and (ii) for $\alpha \ge \frac{6}{\rho}$.

\subsection{Proof of Corollary \ref{cor:approx}}\label{subsec:technical_0}

We will retain all notations used in the proof of Theorem \ref{thm:sparsity}. In the considered case,
a compactly supported function $\psi \in L^2(\R^2)$ can be chosen so that shearlets are defined by
\[
\psi_{\lambda} = |\mbox{det}(A_j)|^{1/2}\psi(S_kA_j\cdot - \mbox{diag}(c_1,c_2)m)
\]
for $\lambda \in \Lambda_s$ with $s = \WQ{s(\lceil j_0 /2\rceil,q_0)} \in \mathbb{S}$ and $k \in \Z$ with \WQ{$s = \tfrac{k}{2^{\lceil j/2 \rceil}}$}.
We emphasize that the additional oversampling matrix $D_p$ is not needed and the index set $\Lambda_s$ originally
defined in Definition \ref{def:dualizableshearlets} is given as
\[
\Lambda_s = \{\lambda = (j,s,m) : j \ge j_0, m \in \Z^2\}.
\]
Further, the shearlet generator $\psi$ can be chosen so that it satisfies a directional vanishing moment condition
(compare Lemma \ref{lemm:technical_4}) in frequency, and that the shearlets $\psi_{\lambda}$ satisfy a support
condition (compare Lemma \ref{lemm:technical_3}(i)) with $\rho = 0$, which yields
\begin{equation}\label{eq:new_support}
\mbox{supp}(\psi_{\lambda}) \subset S^{-1}_sA^{-1}_j([-C,C]^2+m) \quad \mbox{for some}\,\, C>0.
\end{equation}
These two conditions imply Lemma \ref{lemm:technical_5} (i)--(ii) with $\rho = 0$, which can be derived by using
similar arguments as in the proof of Proposition 2.2 in \cite{KL11}.

Now consider index sets $\Lambda^0_j$, $\Lambda^0_{j,r}$, $\tilde{\Lambda}^r_j$ and $\tilde{\Lambda}^0$ as defined in \eqref{eq:index01} -- \eqref{eq:subindex}
for $r = 0,1$. Notice that for those index sets, the additional index $p \in \N_0$ for the sampling matrix $D_p$ is not needed, and we have
\[
\Lambda^0_j = \{\lambda = (j',s,m) \in \Lambda^0 : j' = j\}
\]
as well as
\[
\Lambda^{0}_{j,r} = \{\lambda = (j,s,m) \in \WQ{\Lambda^0_j} : \mathrm{int}(\supp(\psi_{\lambda})) \cap \mathrm{int}(Q_{j,\ell}) \cap \Gamma \neq \emptyset, \,\,
\text{for} \,\, Q_{j,\ell} \in \mathcal{Q}^{r}_j\}
\]
for $r = 0,1$. Moreover, for $j < \WQ{\lceil J/4 \rceil}$, we have $\tilde{\Lambda}^0_j = \Lambda^0_j$ and, for $j \ge \WQ{\lceil J/4 \rceil}$,
\[
\tilde{\Lambda}^0_j = \{\lambda = (j,s,m) \in \WQ{\Lambda^0_j} : \mathrm{int}(\supp(\psi_{\lambda})) \cap \mathrm{int}(Q_{j,\ell}) \cap \Gamma \neq \emptyset,
|k_{j,s} + \WQ{2^{\lceil j/2 \rceil}}\hat{s}_{j,\ell}| \leq 2^{\frac{J-j}{4}}  \hspace*{-0.1cm} , Q_{j,\ell} \in \WQ{\mathcal{Q}^0_j} \}.
\]
Also,
\[
\tilde{\Lambda}^1_j = \{\lambda = (j,s,m) \in \WQ{\Lambda^0_j} : \mathrm{int}(\supp(\psi_{\lambda})) \cap \mathrm{int}(Q_{j,\ell}) \cap \Gamma \neq \emptyset \,\,
\text{for} \,\, Q_{j,\ell} \in \mathcal{Q}^1_j \}.
\]

Applying the same estimate as \eqref{eq:counting} with those index sets $\tilde{\Lambda}^r_j$, we obtain
\begin{equation*}
\#(\tilde{\Lambda}^0) \lesssim J2^{\frac{J}{2}}
\end{equation*}
Note that $\rho = 0$ in \eqref{eq:counting}, since the additional index $p \in \N_0$ is not required for $\tilde{\Lambda}^r_j$ in this case.
Applying \eqref{eq:estimate01}---\eqref{eq:estimate05} with \eqref{eq:new_support} and Lemma \ref{lemm:technical_5} (i)--(ii) with $\rho = 0$, we obtain
\[
\|f-f_N\|^2_2 \lesssim \sum_{\lambda \notin \Lambda_N} |\langle f,\psi_{\lambda}\rangle|^2 \lesssim
2^{-J}
\]
with $N \sim J2^{\frac{J}{2}}$. This proves our claim.

\end{document}